\documentclass[reqno]{amsart}
\usepackage{amsmath,amssymb}
\usepackage[mathscr]{euscript}

\newtheorem{theorem}{Theorem}[section]
\newtheorem{lemma}[theorem]{Lemma}
\newtheorem{proposition}[theorem]{Proposition}
\newtheorem{corollary}[theorem]{Corollary}

\newtheorem{definition}[theorem]{Definition}

\numberwithin{equation}{section}

\newcommand{\tq}{\widetilde{\mc Q}}

\newcommand{\jn}{J^{\text{neum}}}
\newcommand{\mc}[1]{{\mathcal #1}}

\newcommand{\mf}[1]{{\mathfrak #1}}
\newcommand{\mb}[1]{{\mathbf #1}}
\newcommand{\bb}[1]{{\mathbb #1}}
\newcommand{\bs}[1]{{\boldsymbol #1}}
\newcommand{\<}{\langle}
\renewcommand{\>}{\rangle}

\renewcommand{\epsilon}{\varepsilon}
 
\newcommand{\ca}{\mathcal{A}}
\newcommand{\cb}{\mathcal{B}}

\newcommand{\cd}{\mathcal{D}}

\newcommand{\cl}{\mathcal{L}}
\newcommand{\cm}{\mathcal{M}}

\newcommand{\B}{\mathbb{B}}

\newcommand{\J}{\mathbb{J}}
\newcommand{\I}{\mathbb{I}}

\newcommand{\M}{\mathbb{M}}
\newcommand{\N}{\mathbb{N}}

\newcommand{\R}{\mathbb{R}}

\newcommand{\V}{\mathbb{V}}

\newcommand{\Z}{\mathbb{Z}}
\newcommand{\LL}{\mathbb{L}}
\newcommand{\Ss}{\mathbb{S}}

\let\b=\beta
\let\d=\delta

\let\ve=\varepsilon

\let\g=\gamma
\let\s=\sigma

\let\ve=\varepsilon
\let\l=\lambda

\let\n=\eta
\let\r=\rho

\let\G=\Gamma
\renewcommand{\L}{\Lambda}

\newcommand{\1}{\,\rlap{\small 1}\kern.13em 1}

\newcommand{\sqr}[2]{{\vcenter{\hrule height.#2pt%
                      \hbox{\vrule width.#2pt height#1pt\kern#1pt%
                            \vrule width.#2pt}%
                      \hrule height.#2pt}}}
\newcommand{\cqfd}{\hfill$\mathchoice\sqr46\sqr46\sqr{1.5}2\sqr{1}2$\par}

\renewcommand{\limsup}{\mathop{\overline{\hbox{\rm lim}}}}
\renewcommand{\liminf}{\mathop{\underline{\hbox{\rm lim}}}}

\newcommand{\rar}{\rightarrow}

\newcommand{\vecte}{{e}}

\newcommand{\Es}{\mathbb{E}}
\newcommand{\Pb}{\mathbb{P}}

\newcommand{\tbw}{{\widetilde {\mb W}}}

\input color

\begin{document}

\title[Empirical current for Kawasaki with Neumann Kac interaction] {Large deviations of the empirical current for
the boundary driven Kawasaki process with long range interaction}

\author[M. Mourragui]{Mustapha Mourragui}  \thanks{supported by ANR-2010-BLANC-0108,
Projet risc.} \address{Mustapha Mourragui
  \hfill\break\indent LMRS, UMR 6085, Universit\'e de Rouen,
  \hfill\break\indent Avenue de l'Universit\'e, BP.12, Technop\^ole du
  Madril\-let, \hfill\break\indent
F76801 Saint-\'Etienne-du-Rouvray, France.}
\email{Mustapha.Mourragui@univ-rouen.fr}

\date{\today}




\begin{abstract}  
We consider a lattice gas evolving in a bounded cylinder of length
$2N + 1$ and interacting via a Neuman Kac interaction of range $N$, in contact
with particles reservoirs at different densities. We investigate the associated
law of large numbers and large deviations of the empirical current
and of the density. The hydrodynamic limit for the empirical density, obtained in the diffusive
scaling, is given by a nonlocal, nonlinear evolution equation with Dirichlet
boundary conditions.
\end{abstract}

\maketitle
\thispagestyle{empty}

\section{Introduction}
\label{sec1}
The large deviations principle is an inportant topic of interest for the study
of macroscopic properties of non-equilibrium systems. In the last years, many papers have been devoted to the subject.
We just quote a few of them where the issue is addressed in the context of lattice gas dynamics for which large deviation principles can be derived in
the hydrodynamic scaling, \cite{bdgjl10, bd, Der, bdgjl7} and references therein.
Typical examples are systems in contact with two thermostats at different temperatures or with two reservoirs at 
different densities.
A mathematical model for such systems is provided by reversible systems of hopping dynamics combined with
the action of an external mechanism of creation and annihilation of particles, modeling the
exchange reservoirs. 
The action of the reservoirs makes the full process non reversible. A principal generic feature of these systems is that 
they exhibit long range correlations in their steady state.

In this paper we consider a microscopic conservative system, with long range interaction with open boundaries. 
The system is contained in a cylinder $\L_N=\{-N,\cdots,N\}\times\bb T^{d-1}_N$ of length $2N+1$ with axis in direction $u_1$, 
with $\bb T_N^{d-1}$ the  $(d-1)$-dimensional microscopic torus of length $2N+1$ and $N$ a scaling parameter, namely we 
impose periodic boundary conditions in all directions but $u_1$. 
In the bulk, particles evolve according to conservative dynamics (Kawasaki) perturbed by a modified 
version of Kac potential which we call {\sl Neuman Kac potential}.
The Kac potentials $J_N$ are two-body interactions with range $N$ and strength $N^{-d}$: $J_N (u) = N^{-d} J (u)$, $u\in\R^d$, 
where $J$ is a smooth function with compact support.
They have been introduced in \cite{kuh}, and then generalized in \cite{lp}, to present a rigorous derivation 
of the van der Waals theory of a gas-liquid phase transition. 
There have been many interesting results on Kac Ising spin systems in equilibrium 
statistical mechanics. We refer for a survey to the book \cite{pr}.
The so called {\sl Neuman Kac potential}, $\jn_N (u) = N^{-d} \jn (u)$, $u\in\R^d$ (see \eqref{tm1} below) is the modification of the Kac potential that
takes into account the fact that the particles are confined in a bounded domain.

Given $\beta\ge 0$ and a chemical potential $\lambda \in \R$, we consider the Hamiltonian
\begin{equation*}
H_N^\b(\eta)= - \b \sum_{x,y\in\Lambda_N}  \jn_N(x,y) \eta (x)\eta(y)\, +\, \, \lambda \sum_{x\in\Lambda_N}\eta(x)\, ,
\end{equation*}
where $\eta = (\eta (x)\, ,\;  x \in \Z )$, $\eta(x) \in {0, 1}$; $\eta(x) = 1$ if there is a particle at site $x$ and 
$\eta(x) = 0$ if site $x$ is empty. 
One can construct in a standard way an
evolution conserving the total number of particles, the so-called Kawasaki dynamics, which can
be described as follows. Particles attempt to jump to nearest neighbour sites at rates depending
on the energy difference before and after the exchange, provided the nearest neighbour target
sites are empty; attempted jumps to occupied sites are suppressed. The rates are chosen in such
a way that the system satisfies a detailed balance condition with respect to a family of Gibbs
measures, parametrized by the so-called chemical potential $\lambda\in \R$ and fixed $\beta$.
To model the presence of the reservoirs, we superimpose at the boundary to the bulk dynamics
a birth and death process. For a fixed smooth function $b(\cdot)$ defined on the boundary of the domain,
the rates of this birth and death process are
chosen so that a Bernoulli product measure of varying parameter $b(\cdot)$
is reversible for it. This latter dynamics is of course not conservative and keeps the fixed
value of the density equal to $b(\cdot)$ at the boundary. 
This dynamics defines an {\it irreducible}  Markov  jump process on a finite state space;
its stationary measure  $\mu_N^{stat, b(\cdot)}$  is unique.
There is a flow of density through the full system and $\mu_N^{stat, b(\cdot)}$
encodes its long time behavior. The full dynamics is reversible  only if $ \beta=0$ and $ b(\cdot)$ is constant.
We introduce the empirical density $\pi_t^N$ of particles and the integrated empirical current $\mb W_t^N$, 
which measures the total net flow of particles in the time interval
$[0, t]$, associated to a trajectory $(\eta_\cdot)$. 

We analyze here the behavior as $N \uparrow \infty$ of the system when the time is rescaled by $N^2$ ({\it diffusive} {\it limit}).
Our purpose is to investigate the behavior of the current of particles. Problems of this kind have been studied in 
\cite{bdgjl2} and in \cite{bl2}. In both documents the large deviations rate functionals are convex. 
The paper \cite{bdgjl2}, studied
the simple exclusion process, in the torus with periodic
conditions. The paper \cite{bl2} is concerned by the reaction diffusion process, in a one-dimensional interval 
with two types of currents (conservative and non
conservative); some conditions on the convexity on the functionals were imposed.
Our goal is to extend these results to the $d$-dimentional boundary driven systems with
long range interactions, for which the dynamical large deviations functionals are
non-convex.

For important classes of models, the hydrodynamic limit and dynamical large deviations for the empirical density
have been proven, see for example \cite{kov, qrv} for equilibrium dynamics and \cite{bdgjl7, bd, blm} in nonequilibrium dynamics.
For Kawasaki dynamics with Kac potential, the law of large numbers for the empirical density has been proved
on the torus with periodic boundary conditions in \cite{gl},
on the whole lattice in \cite{mm}, and finally on a one-dimensional bounded interval (boundary driven) in \cite{mo3}.
The hydrodynamic equation obtained for the boundary driven dynamics is the following nonlocal, nonlinear partial differential equation 
with Dirichlet conditions at the boundary $\Gamma$ of the domain,
\begin{equation}\label{eq:1}
\begin{cases}
{\displaystyle
\partial_t \rho_t\, =\, 
\nabla \cdot \Big\{ \nabla \rho_t - \beta \sigma (\rho_t) \nabla (\jn  \star \rho_t)   \Big\}
=\, 
-\nabla \cdot \Big\{ {\dot {\mb J}^\b}(\rho_t)  \Big\}
}\\
{\displaystyle
\vphantom{\Big\{}
\rho_t {\big\vert_\Gamma} =\; b (\cdot)  \quad 
\text {for } \ \ 0\le t\le T \;,
} 
\\
{\displaystyle
\rho_0(u) =\g (u) \; ,
}
\end{cases}
\end{equation}
where $*$ stands for the spatial convolution and $\sigma(\rho)=2\rho(1-\rho)$ is the mobility of the system.
In the above formula ${\dot {\mb J}^\b}(\rho_t)$ is the instantaneous current
at time $t$ associated to the trajectory $\rho$:
\begin{equation}\label{J0}
{\dot {\mb J}^\b}(\rho_t) \, =\, -\nabla \rho_t + \beta \sigma (\rho_t) \nabla (\jn  \star \rho_t)\, .
\end{equation}
We shall denote by $\bar \rho$ the unique stationary solution of the hydrodynamic equation, i.e. $\bar \rho$
is the typical density profile for the stationary nonequilibrium state.

It follows from the hydrodynamic limit that
the empirical current $\mb W_t^N$ converges weakly to the time integral of ${\dot {\mb J}^\b}(\rho_s)$ in the time interval 
$[0,t]$ (cf. Proposition \ref{curr}). 
In addition to this we prove that when $\beta$ is small enough, then the empirical particle density $\pi_t^N$ 
obeys a law of large numbers with respect to the stationary measures (hydrostatic), i.e. 
it converges weakly under the unique stationary measure 
of the evolution process to the stationary solution $\bar  \rho$, (see Proposition \ref{th-hy1}). 
This is obtained deriving first the hydrodynamic limit for the empirical density
distributed according to the stationary measure. Then we exploit that the
stationary solution $\bar \rho$ is unique and is a global attractor for the macroscopic evolution. 
Similar strategy for proving the hydrostatic is used in \cite{flm, mo3}.
It then results that, if initially the particles are distributed according to the stationary state $\mu_N^{stat, b(\cdot)}$, 
then for each $t>0$, the mean empirical current $\mb W_t^N/t$ converges weakly to 
${\dot {\mb J}^\b}(\bar \rho)$ as $N\uparrow \infty$ (see Proposition \ref{curr1}). 

\smallskip
Further, we investigate the large deviations for the couple (current, density)= $(\mb W_t^N, \pi_t^N)$, that is
we compute the asymptotic probability of observing an atypical macroscopic trajectory of the (current, density)=
$(\mb W_t,\rho_t)$, when the number of particles tends to infinity.
The result can be informally stated as follows. Given a trajectory $(\mb W_t,\rho_t)_{t\in [0,T]}$ on a fixed 
interval of time $[0,T]$, we have
$$
\mathbb P_{N}^\b \Big( \big(\mb W^N,\pi^N\big) \approx \big(\mb W,\rho\big)\Big)\, \sim\, 
\exp\Big\{ -N^d \mc J_T \big(\mb W,\rho\big)\Big\}\, ,
$$
where $\mathbb P_{N}^\b$ is the law of microscopic dynamics,
$\sim$ denotes the logarithmic equivalence as $N\uparrow \infty$ and $(\mb W^N,\pi^N) \approx (\mb W,\rho)$ 
means that the trajectory $(\mb W^N,\pi^N)$ is in some neighborhood of $(\mb W,\rho)$ for an appropriate
topology.
The rate functional $\mc J_T$ is infinite in the set $\mb E^c$ of all paths $(\mb W,\rho)$ 
that do not satisfy the continuity equation $\partial_t \rho +\nabla \cdot {\dot {\mb W}_t}=0$, and for which some suitable energy estimate does not holds (cf. \eqref{1:Q}). 
Outside this set, 
\begin{equation*}
\label{Ica}
\mc J_T({\mb W},\rho)\;=\; \frac 12 \int_0^T \!dt \,
\Big\langle \big[ {\dot {\mb W}_t} - {\dot {\mb J}^\b}(\rho_t) \big], 
\frac{1}{\sigma(\rho_t)}
\big[ {\dot {\mb W}_t}  - {\dot {\mb J}^\b}(\rho_t) \big] \Big\rangle\, ,
\end{equation*}
where ${\dot {\mb W}_t}$ is the instantaneous current
at time $t$, $\langle \cdot,\cdot\rangle$ denotes
integration with respect to the space variables and ${\dot {\mb J}^\b}(\cdot)$ is defined in \eqref{J0}. 

Our proof relies on the method developed to study hydrodynamic large deviations for the density in \cite{kov, qrv, blm}
and for the current \cite{bdgjl2}.
The basic strategy of the proof of the lower bound consists of two steps, we first obtain this bound for smooth paths, 
then we extend it for general trajectories by showing that, for any given trajectory
$(\mb W, \rho)$ with   finite rate functional   $ \mc J_T (\mb W,\rho)$  
one constructs a sequence of smooth paths 
$(\mb W^n,\rho^n)$ so that $ (\mb W^n,\rho^n )\to (\mb W,\rho)$ in a suitable topology and  
$\mc J_T (\mb W^n,\rho^n) \to  \mc J_T (\mb W,\rho)$. 
The proof in \cite{bdgjl2} relies on the convexity of the rate functional.
In the present case, because of the lack of convexity  we modify the definition of 
the rate functional  declaring it infinite in the set $\mb E^c$.
The modified rate functional $\mc J_T$ makes the proof of the lower and upper bounds harder than the one in \cite{bdgjl2}.  

The last result of this paper is the large deviations for the empirical density. In
one dimension, it has been done in \cite{mo3}. In our context, one can achieve the proof either
following the same scheme as in \cite{mo3}, or adapting the strategy of \cite{bdgjl2}, using the
contraction principle.

\smallskip
The paper is organized as  follows. In section \ref{sec2}, we introduce the model and state the main results.
In Section \ref{mlswd},  we introduce the perturbed model, we prove the law of large numbers for the current,  and we collect
some basic estimates needed along the paper. In Section \ref{secprf}, 
we state and prove some properties of the rate functionals. In sections \ref{secldub} and \ref{secldlb}, 
we derive the upper and lower bounds large deviations for the couple (current, density).  Finally the density large 
deviations are recovered using the contraction principle in section \ref{ldp-ed}.

\medskip
\section{Notation and Results}
\label{sec2}


Fix a positive integer $d\ge 2$.  Denote by $\Lambda$ the open set $(-1,1)
\times \bb T^{d-1}$ and by $\overline \L = [-1,1]
\times \bb T^{d-1}$ its closure, where $\bb T^{k}$ is the $k$-dimensional torus
$[0,1)^k$, and by $\Gamma=\partial\Lambda$ the boundary of $\Lambda$: $\Gamma = \{(u_1,
\dots , u_d)\in \overline \L  : u_1 = \pm 1\}$.

We introduce  a smooth, symmetric,   translational invariant probability kernel of range 1 on
$\Ss_d=\R \times \bb T^{d-1}$, that is, a function $J:\Ss_d\times\Ss_d \to [0,1]$ such that
    $J(u,v) =J(v,u) =  J(0, v-u)$ for all 
$u,v\in \Ss_d$, $J(0,\cdot)$ is continuously differentiable, $J(0,u)=0$, for all $u$ such that $|u_1|>1$, and $ \int J (u,v) dv =1$, for all 
$u\in \Ss_d$.
This is the so called the Kac interaction on $\Ss_d$.

The Neuman Kac interaction $J^{neum}$ 
is  a symmetric  probability kernel on $\L$ defined
by imposing a reflection rule:
when   $ (u, v)  \in \overline\Lambda \times \overline\Lambda $, 
$u$     interacts with  $v$  and  with the reflected points of  $v$
 where reflections are the ones  with    respect  to  the left and right  boundary of $\L$. That is
for all $u$ and $v$ in $\overline\Lambda$ 
 \begin{equation}
\label{tm1}
J^{neum}(u,v):= J(u,v) + J(u,v+ 2(1-v_1)\vecte_1) + J(u,v -2(1+v_1)\vecte_1)\, , 
\end{equation} 
where $v_1$ stands for the first cordinate of the vector $v=(v_1,\cdots,v_d)$ and 
$\{\vecte_1,\ldots,\vecte_d\}$ 
stands for the canonical basis of $\R^d$.

For an integer $N\ge 1$, denote by $\bb T_N^{d-1}=\{0,\dots,
N-1\}^{d-1}$, the discrete $(d-1)$-dimensional torus of length $N$.
Let $\L_N=\{-N,\ldots,N\} \times \bb T_N^{d-1}$ be the
cylinder in $\bb Z^d$ of length $2N+1$ and basis $\bb T_N^{d-1}$ and let
$\G_N=\{(x_1, \dots, x_{d}) \in \bb Z\times \bb T_N^{d-1}\,|\, x_1 =\pm
N\}$ be the boundary of $\L_N$.  The elements of $\L_N$
are denoted by letters $x,y$ and the elements of $\overline\L$ by the
letters $u, v$.

The configuration space
is $\Sigma_N:=\{0,1\}^{\Lambda_N}$; elements of $\Sigma_N$ are denoted
by $\eta$ so that $\eta(x)=1$, (resp.\ $0$)  if site $x$ is occupied,
(resp.\ empty) for  the configuration $\eta$. 


Fix a positive parameter $\beta\ge 0$, and a positive function $b: \Gamma \to \bb R_+$. Assume that there
exists a neighbourhood $V$ of $\overline\L$ and a smooth function $\theta : V
\to (0,1)$ in $\mc C^2(V)$ such that $\theta$ is bounded below by a
strictly positive constant, bounded above by a constant smaller than
$1$ and such that the restriction of $\theta$ to $\Gamma$ is equal to $b$. 
The boundary driven Kawasaki process with Neuman Kac interaction is the Markov process on $\Sigma_N$ whose generator 
$\mf L_{N}\, :=\, \mf L_{\b,b,N}$ 
can be decomposed as
\begin{equation}\label{eq:gen}
\mf L_N \, :=\, N^2\cl_{\beta,N} \;+N^2\; L_{b,N} \, .
\end{equation}

The generator $\cl_{\b,N}$  describes the bulk dynamics which preserves the total number of particles.
The    pair interaction  between  $x$ and $y$ in $\L_N $  is  given by 
$$J_N(x,y)=   N^{-d}J^{neum} (\frac x N , \frac y N ). $$
The total interaction energy among   particles is defined by the following   Hamiltonian 
\begin{equation}
\label{Ha1}
H_N(\eta)= - \sum_{x,y\in\Lambda_N}  J_N(x,y) \eta (x)\eta(y)\, .
\end{equation}
The action of $\cl_{\b,N}$ on functions $f: \Sigma_N\to \R$ is then given by
$$
\left(\cl_{\b,N}  f\right) (\eta)= 
 \sum_{i=1}^d \sum_{x,x+\vecte_i\in\L_N} C_N^\b ({x,x+\vecte_i};\eta)  \left[ f(\eta^{x,x+\vecte_i}) -
f(\eta) \right]\, ,
$$
whith the  rate of exchange occupancies   $C_N^\b $  given by 
\begin{equation}
\label{rate1}
C_N^\b(x,y;\eta)= 
\exp \left  \{ -\frac {\beta}2  [H_N (\eta^{x,y}) -H_N (\eta)]  \right\}\; ,
\end{equation}
where 
$\eta^{x,y}$ is the configuration obtained from $\eta\in \Sigma_N$,  by
exchanging the occupation variables $\eta (x)$ and $\eta (y)$, i.e.\ 
\begin{equation*}
(\eta^{x,y}) (z) := 
  \begin{cases}
        \eta (y) & \textrm{ if \ } z=x\, ,\\
        \eta (x) & \textrm{ if \ } z=y\, ,\\
        \eta (z) & \textrm{ if \ } z\neq x,y\, .
  \end{cases}
\end{equation*}

The generator $L_{b,N}$ models the particle reservoir at the boundary  of $\L_N$,  it is defined by
the infinitesimal generator of a birth and death process acting on $\Gamma_N$ as
\begin{eqnarray*}
(L_{b,N} f)(\eta) \;=\;  \: \sum_{x \in \Gamma_N}
r_x\big(b(x/N),\eta \big) \big[ f(\sigma^{x} \eta)-f(\eta)\big]  \, ,
\end{eqnarray*}
where $\sigma^{x}\eta$ is the configuration obtained from
$\eta$ by flipping the configuration at $x$, i.e.\
\begin{equation*}
(\sigma^{x} \eta) (z) := 
  \begin{cases}
        1-\eta (x) & \textrm{ if \ } z=x\\
        \eta (z) & \textrm{ if \ } z\neq x\, ,
  \end{cases}
\end{equation*}
and for $x\in \Gamma_N$ and $\lambda\in (0,1)$ the rate
$r_x\big(\lambda,\eta)$ is given by
\begin{equation}\label{rate-b}
  r_x\big(\lambda,\eta)\;:=\;
\lambda (1-\eta(x)) + (1-\lambda) \eta(x)\, . 
\end{equation}

For any $\b\ge 0$, the operator $\cl_{\b,N}$ is  self-adjoint w.r.t. the 
Gibbs measures $\mu^{\b,\l}_{N}$ associated to the Hamiltonian \eqref {Ha1}
and chemical potentials $\l\in \R$:
\begin{equation*}
\mu^{\b,\l}_{N}(\eta)
= \frac 1{Z^{\b, \l}_{N}} \exp\{- \beta H_N(\eta) +\l \sum_{x\in \L_N}\eta(x) \}\;, \qquad \eta \in  \Sigma_N\; ,
\end{equation*}
where ${Z^{\b,\l}_{N}}$ is the normalization constant.
 This means that the   rates of the bulk dynamics   
$ \{ C_N^\b(x,y;\eta), \quad x,\,  y  \in \L_N\} $,  
satisfies the detailed balance conditions:
$$ C_N^\b(x,y;\eta)=   e^{- \beta   [H_N (\eta^{x,y}) -H_N (\eta)] } C_N^\b(y,x;\eta^{x,y}). $$

For a smooth function $ \rho: \L \to    (0,1)$ and $x\in\L_N$,
let $\nu_{\rho(\cdot)}^N$ be the Bernoulli
product measure on $\Sigma_N$ with marginals given by
\begin{equation*}
\nu_{\r(\cdot)}^N(\n(x)=1) = \r(x/N)\;.
\end{equation*}
Let $\varphi(\rho(x/N)) := \log [ \rho(x/N)/ (1-\rho(x/N))]$ be
the chemical potential of the profile $\rho(\cdot)$ at site $x/N$. 
It is easy to see that, 
$ \nu_{\rho(\cdot)}^{N}$  can be rewritten as
\begin{equation*}
\label{imr}
  \nu_{\rho(\cdot)}^{N}(\eta) = \prod_{x\in \L_N} 
\frac{ e^{{\varphi}(\rho(x/N)) \, \eta(x)}}{1+e^{{\varphi}(\rho(X/N))}}\;,
\end{equation*}
and if $\rho(u)=b(u)$ for all $u\in\Gamma$, then  $\nu_{{\r}(\cdot)}^N$
is reversible for the process with generator $L_{b,N}$.  

Notice that in view of the
diffusive scaling limit, the generator has been speeded up by $N^2$.
We denote by $(\eta_t)$ the Markov process on $\Sigma_N$ with generator
$\mf L_{N}$.
Since the Markov process $(\eta_t)$ is irreducible, for each $N\ge 1$,
$\b \ge 0$, there exists a unique invariant
measure $\mu^{stat}_N=\mu^{stat}_N(\b,b(\cdot))$ in which we drop the dependence on $\beta$ and $b(\cdot)$ 
from the notation. 
Moreover, if $b(\cdot)$ is not constant then
the invariant measure $\mu^{stat}_N$ cannot be written in simple form.

For an integer $1\leq m\leq +\infty$ denote by  $\mc C^m (\Lambda)$ the space of
$m$-continuously differentiable real functions defined on $\overline \Lambda$.
Let $\mc C^m_0 (\Lambda)$ (resp. $\mc C^m_c (\Lambda)$), $1\leq m\leq
+\infty$, be the subset of functions in $\mc C^m (\Lambda)$ which
vanish at the boundary of $\Lambda$ (resp. with compact support in
$\Lambda$). 
We denote by $\mc M=\mc M (\L)$ the space of finite signed measures on $\L$, 
endowed with the weak topology. For a finite signed measure $m$ and a continuous function $F\in \mc C^0(\L)$,
we let $\< m, F\>$ be the integral of $F$ with respect to
$m$. 

For each configuration $\eta$, denote by $\pi^N = \pi^N(\eta)\in \mc M$ the
positive measure obtained by assigning mass $N^{-d}$ to each particle
of $\eta$~:
\begin{equation*} 
\pi^N \, =\, N^{-d}\sum_{x\in\L_N}\eta(x)\, \delta_{x/N}\; ,
\end{equation*}
where $\delta_{u}$ is the Dirac measure concentrated on $u$. 
Notice that for each $\eta\in \Sigma_N$, the total mass of the positive measure  
$\pi^N(\eta)$ is bounded by 3.

For $t\ge 0$ and two neighboring sites $x,y\in \L_N$,  denote by ${\mb N}^{x,y}_t$ the total number of particles 
that jumped from $x$ to $y$ in the macroscopic time interval $[0,t]$. For $1\le j\le d $ and 
$x,x+{\vecte_j} \in \L_N$, we denote by 
$W_t^{x,x+\vecte_j}= {\mb N}^{x,x+\vecte_j}_t - {\mb N}^{x+\vecte_j,x}_t$ the current through the edge
$(x,x+\vecte_j)$. We now define the current entering and leaving the system through the border points.
For $x\in \Gamma_N$,
let ${\mb N}^{x,+}$ (resp. ${\mb N}^{x,-}$) be the number of particles created (resp. killed) at $x$ due to the reservoir
in the macroscopic time interval $[0,t]$, the current through $x\in\Gamma_N$ is then defined by 
$W_t^{x} = {\mb N}^{x,-}-{\mb N}^{x,+} $.

For $t\ge 0$, we define the \emph{empirical current} 
$\mb W^N_t =(W^{N}_{1,t}, \dots, W^{N}_{d,t}) \in \mc M^d = \{\mc M(\L)\}^d$ as the vector-valued finite signed measure on $\L$
induced by the net flow of particles in the time interval $[0,t]$:
\begin{equation}
\label{cur-emp}
\left\{ 
\begin{array}{l}
\displaystyle W^{N}_{1,t} \;=\; \frac1{N^{d+1}}\sum_{x,x+\vecte_1\in\L_N} 
W^{x,x+\vecte_1}_t \delta_{x/N} \, +\, 
\frac1{N^{d+1}} \sum_{x\in\Gamma_N}W^{x}_t \delta_{x/N}\, ,\\
\displaystyle W^{N}_{k,t} \;=\; \frac1{N^{d+1}} \sum_{x\in\L_N} 
W^{x,x+\vecte_k}_t \delta_{x/N}\; \quad \text{for}\qquad  k=2,\dots,d\, . \\ 
\end{array}
\right. 
\end{equation}
For a continuous vector field $\mb G = (G_1, \dots, G_d) \in (\mc C^0(\L))^d$ the integral of $\mb G$ with respect to $\mb W^N_t$, also
denoted by $\<\mb W^N_t, \mb G\> $, is given by
\begin{equation}
\label{empcurr}
\< \mb W^N_t , \mb G\> =\, \sum_{k=1}^d \<W^N_{k,t}\, ,\, G_k \>\;,
\end{equation}
where
$$
\<W^N_{1,t}\, ,\, G_1 \>\, =\, N^{-(d+1)} \Big\{ \sum_{x,x+\vecte_1\in\L_N} 
G_1 (x/N) \, W^{x,x+e_1}_t 
\, +\, \sum_{x\in\G_N} G_1 (x/N) \, W^{x}_t \Big\}
$$
and for $2\le k\le d$, 
$$
\<W^N_{k,t}\, ,\, G_k \>\, =\,
N^{-(d+1)} \sum_{x\in\L_N} G_k (x/N) \, W^{x,x+e_k}_t\, .
$$

The purpose of this article is to prove hydrodynamic limit and large deviations
for the empirical current and for the density of particles. 
Fix $T>0$. Let $\mc{F}^1$ be the subset of $\mathcal{M}$ of all
absolutely continuous positive measures with respect to the Lebesgue measure
with positive density bounded by $1$:
\begin{equation*}
\mc{F}^1=\big\{\pi\in\mc{M}:\pi(du)=\rho(u)du \;\; \hbox{ and } \;\;
0\leq \rho(u)\leq 1\;  \hbox{ a.e.} \big\}\, .
\end{equation*}
For a metric space $E$ ($E={\mc M},\mc{F}^1 ,{\mc M}^d,\Sigma_N,\cdots$),
let $D([0,T],E)$ be the set of right continuous with left
limits trajectories with values in $E$,  endowed with the Skorohod
topology and   equipped with its Borel $\s-$ algebra.     
For a probability measure  $\mu_N$ on $\Sigma_N$ denote by   $(\eta_t)_{t\in [0,T]}$
the Markov process   with generator 
$\mf L_{N} $ starting,  at time $t=0$, by $ \eta_0$ distributed according to $\mu_N$.  
Denote by $\Pb_{\mu_N}^{\b}:=\Pb_{\mu_N}^{\b,N}$ the 
probability measure on the path space $D([0,T],\Sigma_N)$ corresponding to 
the Markov process $(\eta_t)_{t\in [0,T]}$
and by 
$\Es_{\mu_N}^{\b}$ the expectation with respect to 
$\Pb_{\mu_N}^{\b}$.
 When $\mu_N=\delta_{\eta^N}$ for some configuration $\eta^N\in\Sigma_N$, we write simply
$\Pb_{\eta^N}^{\b}=\Pb_{\delta_{\eta^N}}^{\b,N}$ and $\Es_{\eta^N}^{\b}=\Es_{\delta_{\eta^N}}^{\b}$. We denote by   $\pi^N$   the map from $D([0,T],\Sigma_N)$  to $D([0,T],\mc M)$  defined by $ \pi^N (\eta_{\cdot})_t=  \pi^N (\eta_t)$ and by  
  $Q_{\mu_N}^{\b}  =  \Pb^{\beta}_{\mu_N} \circ (\pi^N)^{-1} $ the law of the process $\big( \pi^N (\eta_t) \big)_{t\in [0,T]}$. 
 
\subsection {Hydrodynamics and hydrostatics}\label{2.1}
The hydrodynamic and hydrostatic limits for the empirical measures $\pi^N$ has been proved in one dimension in
\cite{mo3}.
The analysis in all dimension can be deducted from the same strategy. We shall therefore
summarize the results omitting their proofs.

For  integers $n$ and $m$ we denote by $C^{n,m}([0,T]\times \L)$ the space of functions 
$F= F_t(u): [0,T] \times \overline \L \to \R$  with $n$ derivatives in time and $m$ derivatives in space which are continuous up to 
the boundary.  
We denote by  $C^{n,m}_0([0,T]\times\L)$  the subset of $C^{n,m}([0,T]\times\L)$ of 
functions vanishing  at the boundary of $\L$,  i.e. $F_t{\big\vert_\Gamma}\equiv 0$ for all $t \in [0,T]$.
We finaly denote by  $C^{n,m}_c([0,T]\times\L)$  the subset of  $C^{n,m}([0,T] \times  \L)$  of 
functions   with compact support in $ [0,T]\times \L$.  

Let $L^2(\L)$ be the Hilbert space of functions $F:\L
\to \R$ such that $\displaystyle \int_\L | F(u) |^2 du <\infty$ equipped with
the inner product
\begin{equation*}
\<F,G\> =\int_\L F(u) \,  G (u) \, du\; .
\end{equation*}
The norm of $L^2(\L)$ is denoted by $\| \cdot \|_{L_2(\L)}$. 

Let $H^1(\L)$ be the Sobolev space of functions $F$ with
generalized derivatives $\nabla F=\big(\partial_1 F,\cdots,\partial_d F \big)$	
in $L^2(\L)$. $H^1(\L)$ endowed with the scalar product
$\<\cdot, \cdot\>_{H^1}$, defined by
\begin{equation*}
\<F,G\>_{H^1} = \< F, G \> + 
\<\nabla F \, , \, \nabla G \>\;,
\end{equation*}
is a Hilbert space. The corresponding norm is denoted by
$\|\cdot\|_{H^1}$.
Denote by  $H^{1}_0(\Lambda)$ the closure of $C^\infty_c(\L)$ 
in $H^1(\L)$.  

Denote by ${\bf Tr}:H^1(\L) \to L^2(\Gamma)$ the continuous linear
operator called trace operator, defined as the unique extension of the linear operator from $\mc C^0({\L})$ to $L^2(\Gamma)$
which associates to any $F\in H^1(\L)\cap \mc C^0({\L})$ its boundary value: $\text{\bf Tr} (G )=
G\big|_{\Gamma}$ (\cite{z}, Theorem 21.A.(e)).
Recall that the space $H^1_0(\L)$ is the space of functions $F$ in
$H^1(\L)$ with zero trace (\cite{z}, Appendix (48b)):
\begin{equation*}
H^1_0(\L) = \left\{F\in H^1(\L):\; \text{\bf Tr}(F) = 0\right\}\,.
\end{equation*}

To state the hydrodynamic equation, we need some more notation.
For a Banach space $(\bb
B,\Vert\cdot\Vert_{\bb B})$ we denote by $L^2([0,T],\bb B)$
the Banach space of measurable functions $U:[0,T]\to\bb B$ for which
\begin{equation*}
\Vert U\Vert^2_{L^2([0,T],\bb B)} \;=\; 
\int_0^T\Vert U_t\Vert_{\bb B}^2\, dt \;<\; \infty
\end{equation*}
holds. 
For  $m\in L_\infty(\L)$ and $u\in\L$, we set 
$$
(\jn\star m)(u)=\int_\L \jn (u,v)m(v) dv\, ,
$$
and   $\chi(m) = m(1-m)$, $\sigma(m)=2\chi(m)$. For any smooth function $F$, let
$\Delta F$ be the laplacian with respect to the space variable of a function $F$. 
For   $F\in C^{1,2}_0([0,T]\times \L)$, $\rho  \in D([0,T], \mc F^1)$ denote
 \begin{equation}
\label{lb1}
\begin{split}
 &\ell_F^\b(\rho|\rho_0) := \big\langle \rho_T, F_T \big\rangle 
- \langle {\rho_0}, F_0 \rangle
- \int_0^{T} \!dt\, \big\langle \rho_t, \partial_t F_t \big\rangle\\
&\qquad\qquad\quad
\, - \int_0^{T} \!dt\, \big\langle  \rho_t , \Delta F_t \big\rangle
\;+\; \int_0^T dt
\int_{\Gamma}b(r) \,  \text{\bf n}_1(r)\, (\partial_{1}F_t)(r) \, dS(r)  \\
&\qquad\qquad\quad
-  \b\int_0^T \langle \sigma( \rho_t), (\nabla F_t)\cdot \nabla (\jn\star  \rho_t)\rangle  dt\, ,
\end{split}
\end{equation}
where {\bf n}=$(\text{\bf n}_1,\ldots ,\text{\bf n}_d)$ stands for the
outward unit normal vector to the boundary surface $\Gamma$ and
$\text{d} \text{S}$ for an element of surface on $\Gamma$. For $u,v\in \R^d$, 
$u\cdot v$ is the usual scalar product of $u$ and $v$ in $\R^d$, we denote by $|\cdot|$ the associated norm:
$|u|=\sqrt{\sum_{i=1}^d|u_i|^2}$.

 Denote by  $\mc A _{[0, T]} \subset D\big([0,T]; \mc F^1\big)  $ the set of all weak  solutions of the  
 boundary value problem \eqref {eq:1} without fixed initial condition:
\begin{equation*}
\mc A _{[0, T]} = \Big\{  \r \in  L^2\big([0,T],H^1(\Lambda)\big) \; :\; 
\quad \forall  F \in \mc C^{1,2}_0 ( [0,T]
 \times \L) \, ,\; \ell_F^\b(\rho| \rho_0)=0
   \Big\}\, .
\end{equation*}

\begin{proposition}
\label{th-hy}
For any sequence of initial probability measures $(\mu_N)_{N\ge 1}$, the sequence of probability measures $(Q_{\mu_N}^{\b})_{N\geq 1}$ 
is weakly relatively compact and  all its converging
subsequences converge to   some limit  $Q^{\beta,*}$ that is concentrated on absolutely continuous paths  whose densities
$\r\in C ([0,T],\mc F^1(\L))$ are   in $\mc A_{[0, T]}  $. Moreover, if for any $\d>0$ and for  any  function 
$F\in \mc C^0(\L)$
\begin{equation}\label{pfl}
\lim_{N\to \infty}\mu^N\Big\{ \Big| \langle  \pi_N, F\rangle\, -\, \int_\L \g (u)F(u)du   \Big| \ge \delta \Big\}=0\, ,
\end{equation}
for an initial continuous profile $\g:\L\to [0,1]$,  then
the sequence of probability measures $(Q_{\mu_N}^{\b})_{N\geq 1}$ 
converges to the Dirac measure  concentrated on the unique   weak  solution $\rho(\cdot,\cdot)$ of boundary value 
problem \eqref{eq:1}. Accordingly, for any $t\in [0,T]$, 
any $\d>0$ and any function $F\in \mc C^0(\L)$
$$
\lim_{N\to \infty}\Pb_{\mu_N}^{\b}\Big\{ \Big| \langle  \pi_N(\eta_t), F\rangle \, -\, \int_\L \rho(t,u)F(u)du  
 \Big| \ge \delta \Big\}=0\, .
$$
\end{proposition}

The proof of this Proposition is similar to the one of Theorem 2.1. in \cite{mo3}.
Recall that the stationary measure $\mu_N^{stat}$ depends on $\b$ and $b(\cdot)$. The
asymptotic behavior of the empirical measure under the stationary state $\mu_N^{stat}$ can be stated as follows.

\begin{proposition}
\label{th-hy1} There exists $\b_0$ depending on $ \Lambda$ and $\jn$ so that, for any  $\beta < \b_0$,
for any   $F \in  C^{0}(\L)$,  for any $\delta>0$,
 $$  \lim_{N \to \infty} \mu_N^{stat} \Big [ \Big|  \langle  \pi_N(\eta), F\rangle \, -\, 
  \int_\L \bar \rho(u)F(u)du   \Big| \ge \delta\Big] =0 \, ,$$
 where $\bar \rho$ is  the unique   weak  solution of the following
 boundary value problem 
\begin{equation} 
\label{eq:1s}
\begin{cases}
{\displaystyle
 \Delta \rho(u)-   \b \nabla \cdot \Big\{ \sigma( \rho(u)) \nabla (\jn  \star \rho )(u)   \Big\}
= 0, \quad u \in \L,} \\
{\displaystyle
\vphantom{\Big\{}
\rho (\cdot){\big\vert_\Gamma} =\; b (\cdot) \;.
} 
 \end{cases}
\end{equation}

\end{proposition}

The proof of this Proposition is similar to the one of Theorem 2.3. in \cite{mo3} and therefore
is omitted.

We turn now to the asymptotic behavior of the empirical current. Next result states that it  converges to the time integral of
the instantaneous current ${\dot {\mb J}^\b}(\rho_t)$ associated to the solution of the hydrodynamic equation \eqref{eq:1}:

\begin{proposition}
\label{curr}
Fix an initial profile $\g\in \mc F^1$ 
and consider a sequence of probability measures $\mu^N$ 
associated to $\g$ in the sense of \eqref{pfl}. 
Let $\rho$ be the solution of the equation
\eqref{eq:1}. Then, for each $T>0$, $\delta>0$ and $\mb G \in \big(C^1(\L)\big)^d$,
\begin{equation*}
\lim_{N\to\infty} \Pb_{\mu_N}^{\b} \Big[\, \big\langle {\mb W^N_T}, \mb G\big\rangle
\;-\; \int_0^T dt\, \big\langle \big\{-\nabla\rho_t +  \b  \sigma(\rho_t)\nabla (\jn  \star \rho_t ) \big\}
\, , \,\mb G\big\rangle \Big\vert > \delta \Big ] \;=\; 0\; .
\end{equation*}
\end{proposition}

Next result concerns the asymptotic behavior of the mean empirical
current $\mb W^N_T/ T$ under the sequence of stationary measures $\{ \mu_N^{stat}\; :\; N\ge 1 \}$.

\begin{proposition}
\label{curr1}
There exists $\b_0$ depending on $ \Lambda$ and $\jn$ so that, for any  $\beta < \b_0$,
for any  $T>0$, $\delta>0$ and $\mb G \in \big(C^1(\L)\big)^d$,
\begin{equation*}
\lim_{N\to\infty} \Pb_{\mu_N^{stat}}^{\b} \Big[\, 
\big\langle\frac1T{\mb W^N_T}, \mb G\big\rangle
\;-\; \, \big\langle\big\{ -\nabla  \bar\rho +  \b  \sigma(\bar \rho)\nabla (\jn  \star \bar \rho ) \big\}
\, , \,\mb G\big\rangle
\Big\vert > \delta \Big ] \;=\; 0\; ,
\end{equation*}
where $\bar \rho$ is  the unique   weak  solution of the boundary value problem \eqref{eq:1s}.
\end{proposition}
 
The proof of Proposition \ref{curr} is given for more general processes in section  \ref{mlswd}.
We obtain then Proposition \ref{curr1} as an immediate  consequence from Proposition \ref{th-hy1}. 
 
\smallskip
\subsection{Large deviations}
Fix a positive time $T>0$  and an initial profile $\g\in \mc F^1$.
We are interested both on large deviations of the couple $(\mb W_t^N , \pi^N(\eta_t))_{t\in [0,T]}$ and  
on large deviations of the empirical measure $(\pi^N(\eta_t))_{t\in [0,T]}$ during the interval time $[0,T]$ 
and starting from the profile $\g$.

Let $\mf A_\g$ be the set of trajectories
$(\mb W, \pi)$ in $D([0,T], \mc M^{d+1})$ such that
for any $t\in[0,T]$ and any $G \in \mc C_0^1(\L)$ 
\begin{equation}
\label{eq:4.5}
\langle \pi_t , G\rangle - \langle \gamma , G\rangle = 
\langle \mb W_t , \nabla G\rangle  \, .
\end{equation}

Define the energy functional 
${\mc E}^\gamma={\mc E}^{\gamma, T,\b} :D([0,T],{\mc M}^{d+1})\to[0,\infty]$ by
\begin{equation}
\label{1:Q}
{\mc E}^{\gamma}(\mb W,\pi)=
\begin{cases}
 {\mc Q} (\pi) &  \hbox{  if  } \ \ (\mb W,\pi) \in \mf A_\g \cap D([0,T], \mc M^d\times\mc F^1)\, ,\\ 
+ \infty & \hbox{ otherwise,}
\end{cases}
\end{equation}
where the functional ${\mc Q}:D([0,T],\mc F^1)\to[0,\infty]$ is given for a trajectory $\pi \in D([0,T], \mc F^1)$ with
$\pi_t =\rho_t(u) du\, ,\  t\in[0,T$] by the formula
\begin{eqnarray*}
{\mc Q}(\pi) = \sum_{k=1}^d \sup
\Big\{  \int_0^T dt\; \langle \rho_t,\partial_{k}H_t\rangle 
- 2\int_0^Tdt\int_{\L} \sigma (\rho_t(u)) H(t,u)^2\, du \Big\}\, , 
\end{eqnarray*}
in which the supremum is carried over all   $H \in
  C^{\infty}_c([0,T]\times \L)$.   
It has been proved in \cite{blm, flm} that  ${\mc Q}(\pi)$ is finite if and only if 
$\rho\in L^2\big([0,T],H^1(\L)\big)$, and
\begin{equation}\label {tm4}
{\mc Q}(\pi) \;=\; \frac 18 \int_0^T dt\, \int_{\L} du\,
\frac{\big|\nabla \rho_t (u)\big|^2}{\sigma (\rho_t(u))}\;\cdot 
\end{equation}
Notice that $\mf A_\g \cap D([0,T], \mc M^d\times\mc F^1)$ is a closed and convex subset of
$D([0,T], \mc M^{d+1})$. It follows immediately from the concavity of $\s(\cdot)$ that
the functional $ {\mc E}^\gamma$ is convex and lower semicontinuous.   

We now define the large deviations functional for the pair $(\mb W^N, \pi^N)$ in the time interval  $[0,T]$ 
with initial condition $\g$. 
For each $\mb V \in
\big(C^{1,1}([0,T]\times \L)\big)^d$, 
define the functional ${\widehat \J}^T_{\mb V}= {\widehat \J}_{\mb V}^{T,\b} :D([0,T],\mc M^d\times\mc F^1)\to \R$ 
if $\pi_t =\rho_t(u) du\, ,\  t\in[0,T$] by
\begin{equation}\label{eq:4.6}
 {\widehat \J}_{\mb V}^T(\mb W, \pi) \, =\, \mb L_{\mb V}^\b (\mb W, \pi)
\;-\; \frac 12 \int_0^T dt \, \< \sigma(\rho_t),\mb V_t\cdot \mb V_t \>\; ,
\end{equation}
where $\mb L_{\mb V}^\b(\mb W, \pi):=\mb L_{\mb V,T}^\b(\mb W, \pi) $ is a linear function on $\mb V$:
\begin{equation*}
\begin{aligned}
&\mb L_{\mb V}^\b(\mb W, \pi) \,=\, \< \mb W_T, \mb V_T\> \, 
-\, \int_0^T dt \, \< \mb W_t, \partial_t \mb V_t\> \\
&\qquad\qquad
-  \int_0^T dt \, \< \pi_t, \nabla \cdot \mb V_t\>
\;+\; \int_0^T dt
\int_{\Gamma}b(r) \,  \text{\bf n}_1(r)\, V_1(t, r) \, dS(r)  \\
&\qquad\qquad
-  \b\int_0^T \langle \sigma( \rho_t), \mb V_t\cdot \nabla (\jn\star  \rho_t)\rangle  dt\, .
\end{aligned}
\end{equation*}

The large deviations fuctional for $(\mb W^N, \pi^N)$ is finally defined from $D([0,T], \mc M^{d+1})$
to $[0,+\infty]$ by
\begin{equation}
\label{f10}
\mathcal J_T^{\g}(\mb W, \pi) \; =\;
\left\{
\begin{array}{ll} \displaystyle 
\J_T(\mb W, \pi)  & \text{if ${\mc E}^\gamma (\mb W, \pi) < \infty$}\, ,\\
+ \infty & \text{otherwise}\;, 
\end{array}
\right.
\end{equation}
where
$$
\J_T(\mb W, \pi) \, =\, 
\sup_{\mb V \in \big(C^{1,1}([0,T]\times \L)\big)^d}
{\widehat \J}^T_{\mb V}(\mb W, \pi) \, .
$$


It remains to define the rate functional for the empirical measure.
Denote by $\I_T^{\g}=\I_T^{\g,\b}\colon D([0,T], \mc F^1)\longrightarrow \bb [0,\infty]$ 
the functional given for a trajectory $\pi$ with
$\pi_t (du)=\rho_t(u) du\, ,\  t\in[0,T$] by
\begin{equation}\label{Jb1}
\I_T^{\g}(\pi) = \sup_{F\in
  C^{1,2}_0([0,T]\times\L)} {\widehat \I}_F^{T,\g}(\pi)\, ,
\end{equation}
where for any function $F\in {\mc C}_0^{1,2}([0,T]\times \L)$,
${\widehat \I}^{T,\g}_F = {\widehat \I}^{T,\g, \b}_F \colon D([0,T], \mc F^1)\longrightarrow
\bb R$ is given by
\begin{equation*}
{\widehat \I}_F^{T,\g}(\pi) := \ell_F^\b (\rho|\g)\; -\;\frac12 \int_0^{T} \!dt\,
 \big\langle \sigma( \rho_t ), \nabla F_t \cdot \nabla F_t  \big\rangle \; .
\end{equation*}
The definition of $\ell_F^\b (\cdot|\g)$ is given by \eqref{lb1}.

The rate functional $\mc I_T^\g: D([0,T], \mc M) \to
[0,\infty]$ for the empirical measure is then given by
\begin{equation}
\label{3:Ib}
\mc I_T^\g(\pi) =
\begin{cases}
 \displaystyle \I_T^\g(\pi)\ & \hbox{ if } \ \ 
  \pi \in  D([0,T], \mc F^1)\, \text{and}\, \mc Q(\pi)<+\infty\, ,\\ 
+\infty & \hbox{ otherwise .}
\end{cases}
\end{equation}

We are now ready to state the large deviations results:
\begin{theorem}
\label{gd-(d,c)}
Fix $T>0$ and an initial profile $\gamma$ in $\mc C^0(\L)$.  Consider a
sequence $\{\eta^N : N\ge 1\}$ of configurations associated to
$\gamma$ in the sense of \eqref{pfl}. 
Then, for each closed set $\mc C$
and each open set $\mc U$ of $D([0,T], \mc M^{d+1})$, we have
\begin{equation*}
 \begin{aligned}
\varlimsup_{N\to\infty} \frac 1{N^d} \log \bb P_{\eta^N}^\b
\Big[ (\mb W^N, \pi^N) \in \mc C \Big] &\;\le\; - \inf_{(\mb W, \pi) \in \mc C}
{\mc J}_T^\g (\mb W, \pi) \;, \\
\varliminf_{N\to\infty} \frac 1{N^d} \log \bb P_{\eta^N}^\b
\Big[ (\mb W^N, \pi^N) \in \mc U \Big] &\;\ge\; - \inf_{(\mb W, \pi) \in \mc U}
{\mc J}_T^\g (\mb W, \pi) \;.
 \end{aligned}
\end{equation*}
The functional  ${\mc J}_T^\g(\cdot,\cdot)$ is lower semi-continuous.  
\end{theorem}

We prove this Theorem in sections \ref{secldub} and \ref{secldlb}.
We have the following dynamical large deviation principle for the empirical measure.

\begin{theorem}
\label{s02}
Fix $T>0$ and an initial profile $\gamma$ in $\mc C^0(\L)$.  Consider a
sequence $\{\eta^N : N\ge 1\}$ of configurations associated to
$\gamma$ in the sense of \eqref{pfl}.   
Then, the sequence of probability measures $\{ Q_{\eta^N}^{\beta} : N\ge 1\}$ on $D([0,T], \mc M)$
satisfies a large deviation principle with speed $N$ and  rate
function $\mc I_T^\g(\cdot)$, defined in \eqref {3:Ib}:    
\begin{eqnarray*}
&& 
\varlimsup_{N\to\infty} \frac 1{N^d} \log Q_{\eta^N}^{\beta}
\big( \pi^N \in C\big)
\;\leq\; - \inf_{\pi \in C} \mc I_T^\g (\pi)  
\\
&& 
\varliminf_{N\to\infty} \frac 1{N^d} \log Q_{\eta^N}^{\beta}
\big( \pi^N\in U \big) \;\geq\; -  \inf_{\pi \in U} \mc I_T^\g (\pi) \;, 
\end{eqnarray*}
for any   closed set  $C \subset D([0,T], \mc M)$   and    open set $U  \subset D([0,T], \mc M)$.
   The functional  $\mc I_T^\g(\cdot)$    is  lower semi-continuous and has compact level sets.  
\end{theorem}

The proof of this Theorem is given in Section \ref{ldp-ed}. It relies on
Theorem \ref{gd-(d,c)} and the  contraction principle.

\bigskip
\section{The perturbed dynamics and basic tools}\label{mlswd}
In this section, we consider the perturbation of the  original
process \eqref{eq:gen}, and we prove some results   needed either to caracterize the behavior of the empirical current and the
empirical density, either to   prove large deviations principle.  

\subsection{The modified process}\label{s-perturb}
Fix $T>0$, a time dependent vector-valued function $\mb V =(V_1,
\dots, V_d) \in \big(\mc C^{0,0}( [0,T]\times \L)\big)^d$ and asmooth function $H\in \mc C^{0,0}( [0,T]\times \G)$. 
Define at time $t$, $0\le t\le T$,  the following generators of a time inhomogeneous Markov process  on $\Sigma_N$
\begin{equation*}
 \begin{aligned}
\big(\cl_{\b,N}^{\mb V} f\big) (\eta) &\, = \,
 \sum_{i=1}^d\sum_{x,x+\vecte_i \in \L_N} C_{N,t }^{\b, V_i}(x,x+\vecte_i;\eta) \left[ f(\eta^{x,x+\vecte_i}) -f(\eta)\right]\, , \\
(L_{b,N}^H f)(\eta) &\,=\,  \sum_{x \in \Gamma_N}
r_{x,t}^H\big(b(x/N),\eta \big) \big[ f(\sigma^{x} \eta)-f(\eta)\big]  \, ,
 \end{aligned}
\end{equation*}
where the rate function $C_{N,t }^{\b, V_i}(x,x+\vecte_i;\eta)$ is defined through the rate $C_N^\b$ by     
\begin{equation}\label{rate-V}
C_{N,t}^{\b, V_i} (x,x+\vecte_i;\eta) \, = \, 
C_{N}^{\b} (x,x+\vecte_i;\eta) 
e^{-[  \eta(x+e_i)-\eta(x)]N^{-1} V_i(t,x/N)} \, ,
\end{equation}
and the rate at the boundary $r_{x,t}^H\big(b(x/N),\eta \big)$) is defined through the rate $r_x$ as
\begin{equation}\label{rate-H}
r_{x,t}^H\big(b(x/N),\eta \big) \, = \, r_x \big(b(x/N),\eta \big) 
 e^{(2\eta(x) -1)N^{-1} H(t,x/N)}\, .
\end{equation}

For a probability measure  $\mu_N$ on $\Sigma_N$  denote by 
$\Pb_{\mu_N}^{\b,\mb V,H}$ the law of the inhomogeneous Markov process $(\eta_t)_{t\in [0,T]}$ on the path space 
$D\big( [0,T], \Sigma_N\big)$ with generator $\mf L_N^{\mb V,H} =N^2\cl_{\b,N}^{\mb V} +N^2 L_{b,N}^H$ and initial distribution 
$\mu_N$. Let $Q_{\mu_N}^{\b,\mb V,H}$ be the measure of the process $(\pi_t^N)_{t\in[0,T]}$
on the state space $D\big( [0,T],\mc M)$ induced from $\Pb_{\mu_N}^{\b,\mb V,H}$.

\begin{proposition}\label{hy-curr}
 Let $\mu_N$  be a sequence of probability measures
 on $\Sigma_N$  corresponding to  a macroscopic profile $\gamma$ in the sense of \eqref{pfl}. Then 
the sequence of probability measures $Q_{\mu_N}^{\b,\mb V,H}$ converges   as  $N\uparrow \infty$, to $Q^{\b,\mb V}$. 
This limit point is concentrated on the unique weak solution $\rho^{\b,\mb V}$ in $L^2([0,T],H^1(\L))$ of the following boundary value problem  
\begin{equation} 
\label{eq:V}
\begin{cases}
{\displaystyle
\partial_t \rho
+ \,   \nabla \cdot \big\{  \sigma(\rho) \big[\b \nabla (\jn  \star \rho )  \, +\, \mb V\big] \big\}
= \, \Delta \rho}\\
{\displaystyle
\vphantom{\Big\{}
\rho (t, \cdot){\big\vert_\Gamma} =\; b (\cdot)  \quad 
\text {for } \ \ 0\le t\le T \;,
} 
\\
{\displaystyle
\rho_0(u) =\g (u) \; .
}
\end{cases}
\end{equation}
Moreover, for each $t>0$, $\delta>0$ and  $\mb G\in (\mc C^1(\L))^d$, we have
\begin{equation}\label{eq:V2}
\lim_{N\to\infty} \Pb_{\mu_N}^{\b,\mb V,H} \Big[\, \Big\vert 
\left\langle \mb W_t^N, \mb G\right\rangle \;-\; \int_0^t ds\, 
\big\langle {\dot {\mb J}}(\rho^{\b,\mb V}), \mb G \big\rangle\
\Big\vert > \delta \Big ] \;=\; 0\; ,
\end{equation}
where ${\dot {\mb J}}(\rho^{\b,\mb V})$ is  is the instantaneous current
associated to $\rho^{\b,\mb V}$ and is given by
$$
{\dot {\mb J}}(\rho^{\b,\mb V})\, =\, -\nabla \rho^{\b,\mb V}+
   \sigma (\rho^{\b,\mb V}) \big[\b \nabla (\jn  \star \rho^{\b,\mb V} )  \, +\, \mb V\big] \, .
$$

\end{proposition}

\begin{proof}
The identification of the limit for the empirical density $(\pi^N(\eta_t))_{t\in[0,T]}$  is similar to the one
of \cite{mo3}. We therefore switch to the limit \eqref{eq:V2}. Following the same steps  as in \cite{bdgjl2},
we consider the family of jump martingales 
\begin{equation*}
 \begin{aligned}
&{\widetilde W}_t^{x,y} \, =W_t^{x,y}
   -N^2\int_0^t \big[\eta(x) -\eta(y)\big]C_{N,t}^{\b, V_i} (x, y;\eta_s)
  ds\ \textrm{for}\   y=x+e_i\, , \ x, y\in \L_N\, ,\\
&{\widetilde W}_t^{y} \, =W_t^{y}-N^2\int_0^t
\Big\{\eta_s(y)(1-b(y/N))e^{N^{-1}H(s,y/N)} \\
&\qquad\qquad\qquad \qquad\qquad -(1-\eta_s(y))b(y/N)e^{-N^{-1}H(s,y/N)}
\Big\}ds\; ,\qquad  y\in \G_N\, .
\end{aligned}
\end{equation*}

Recall from \eqref{cur-emp} the definition of the empirical measures $(W_{j,t}^N)_{t\ge 0}\; , 1\le j\le d$.
Fix a smooth vector field ${\mb G}=(G_1,\cdots,G_d)\in \big(\mc C^{1,1}([0,T]\times\L))^d$, and  
consider the   $\Pb_{\mu^N}^{\b,\mb V,H}$-martingale 
${\widetilde {\mb W}}_t^{\mb G, \mb V,H}\equiv \tbw_t^{{\mb G},\mb V,H, N,\b}$,
$t\in[0,T]$, defined by 
\begin{equation*}
 \begin{aligned}
&{\widetilde {\mb W}}_t^{\mb G, \mb V,H} \, =\,
\sum_{k=1}^d \Big\{ \big\langle W_{k,t}^N\, , \,  G_k \big\rangle\\
&\ \ 
-\frac{N^2}{N^{d+1}} \sum_{x,x+\vecte_k\in \L_N}  G_k(x/N) \int_0^t \big[\eta_s(x) -\eta_s(x+\vecte_k)\big]
C_{N,t}^{\b,V_k} \big(x,x+\vecte_k; \eta_s\big) ds \Big\}\\
&\ \ 
-\frac{N^2}{N^{d+1}} \sum_{x\in \G_N}  G_1(x/N) \int_0^t 
\Big\{\eta_s(x)(1-b(x/N))e^{N^{-1}H(s,x/N)} \\
&\qquad\qquad\qquad \qquad\qquad \qquad\qquad-(1-\eta_s(x))b(x/N)e^{-N^{-1}H(s,x/N)}
\Big\}ds\, .
 \end{aligned}
\end{equation*}
From Lemma \ref{b1} and  Taylor expansion the integral term of  the last expression is equal to
\begin{equation*}
 \begin{aligned}
&-\frac{N^2}{N^{d+1}} \sum_{k=1}^d\sum_{x,x+\vecte_k\in \L_N} G_k(x/N)
\int_0^t \big[\eta_s(x) -\eta_s(x+\vecte_k)\big]\, ds\\
&\ 
-\frac{1}{N^{d}} \sum_{k=1}^d\sum_{x,x+\vecte_k\in \L_N} G_k(x/N)
\int_0^t \big[\eta_s(x) -\eta_s(x+\vecte_k)\big]^2\Upsilon^{\b,\mb V}_k(\pi^N(\eta_s),s,x/N)\, ds\\
&\ 
-\frac{1}{N^{d-1}} \sum_{x\in \G_N}  G_1(x/N) \int_0^t \big[\eta_s(x) -b(x/N)\big] ds
\, + \, O_{\mb G,\b,\mb V,H}\big(N^{-1}\big)\, ,
\end{aligned}
\end{equation*}
where for $1\le k\le d$, $\eta\in \Sigma_N$, $s\ge 0$ and $x\in\L_N$,
$$
\Upsilon^{\b,\mb V}_k(\pi^N(\eta),s,x/N)\, =\, \b
\partial_k^N( \jn\star \pi^N(\eta_s))(x/N) \,+\,V_k(s,x/N)\, ,
$$
for any smooth function $G$, $\partial_j^N G$ is defined in \eqref{derive1},
and $O_{\mb G,\b, \mb V,H}\big(N^{-1}\big)$ is an expression whose absolute
value is  bounded by $C N^{-1}$ for some constant depending on $\mb G$, $\b$, $\jn$, $\mb V$ and $H$.
A summation by parts and Taylor expansion permit to rewrite the martingale ${\widetilde {\mb W}}_t^{\mb G, \mb V,H}$  as
\begin{equation*}
 \begin{aligned}
&{\widetilde {\mb W}}_t^{\mb G, \mb V,H} \, =\, 
\big\langle \mb W_{t}^N\, , \,  \mb G \big\rangle\,
-\, \frac{1}{N^{d}} \sum_{k=1}^d \sum_{x\in \L_N\setminus\G_N}  \int_0^t  \, ds\, \big(\partial_k G_k\big)(x/N) \eta_s(x) \\
&\ \ 
-\, \frac{1}{N^{d}} \sum_{k=1}^d \sum_{x\in \L_N\setminus\G_N}  \int_0^t  \, ds\, G_k({x}/N)\big[ \eta_s(x) -\eta(x+e_k)\big]^2
\Upsilon^{\b,\mb V}_k(\pi^N(\eta_s),s,x/N)\Big] \\
&\ \ 
-\Big\{\frac{1}{N^{d-1}} 
\sum_{x\in \G_N^-}  G_1(x/N) \int_0^t \eta_s(x) ds \, -\, \frac{1}{N^{d-1}} 
\sum_{x\in \G_N^+}  G_1(x/N) \int_0^t \eta_s(x) ds\Big\}\\
&\ \ 
-\frac{1}{N^{d-1}} \sum_{x\in \G_N}  G_1(x/N) \int_0^t \big[\eta_s(x) -b(x/N)\big] ds
\, + \, O_{\mb G,\b,\mb V,H}\big(N^{-1}\big)\, .
 \end{aligned}
\end{equation*}
Here, $\Gamma_N^-$, resp. $\Gamma_N^+$, stands for the left, resp. right, boundary of $\L_N$:
\begin{equation*}
\Gamma^{\pm}_N = \{(x_1,\cdots,x_d)\in \Gamma_N: x_1 = \pm N\}
\end{equation*}
Next, we use  the replacement lemma stated in 
Proposition \ref{see}. We   obtain that the martingal ${\widetilde {\mb W}}_t^{\mb G, \mb V,H} $ can
be replaced by
\begin{equation*}
 \begin{aligned}
&
\big\langle \mb W_{t}^N\, , \,  \mb G \big\rangle -
\int_0^t ds\Big\{\sum_{k=1}^d <\pi^N(\eta_s), \partial_k G_k>\Big\}\,
+ \, \int_0^t ds
\int_{\Gamma}G_1(r) \, b(r) \,  \text{\bf n}_1(r) \, dS(r)\\
&\qquad\qquad
\, -\,
\int_0^t ds\Big\{\frac1{N^d}\sum_{k=1}^d\sum_{x\in \L_N\setminus\G_N}  G_k(x/N)\sigma \left(\eta_s^{\varepsilon N}(x)   \right) 
\Upsilon^{\b,\mb V}_k(\pi^N(\eta_s),s,x/N) \Big\}\,  .
 \end{aligned}
\end{equation*}
On the other hand, a simple computation shows that the expectation of the quadratic variation of the martingale 
${\widetilde {\mb W}}_t^{\mb G, \mb V,H}$ vanishes as $N\uparrow 0$. Therefore, by Doob's
inequality, for every $\d >0$, 
\begin{equation}
\label{bv}
\lim_{N\rar \infty} \Pb_{\mu_N}^{\b,\mb V,H} \Big[ \sup_{0\le t\le T}
|{\widetilde {\mb W}}_t^{\mb G, \mb V,H} |>\d \Big] \; =\;0 \; . 
\end{equation}
Finally, recall that by the first part of the proposition, the empirical density converges to the solution of the
equation \eqref{eq:V}. This concludes the proof.
\end{proof}
\medskip

\subsection{Some useful tools}
In this section we collect some technical results which will be used in the proof both of
the hydrodynamic limit and of the dynamical large deviation principle.
 We start by some properties of the potential $\jn(\cdot,\cdot)$  easily 
obtained by its definition.

\begin{lemma}\label{lem-jn}
The potential $\jn(\cdot,\cdot)$ is a symmetric probability kernel. Moreover for any regular function 
$F:\L\to \R$ and $1\le k\le d$, 
we have the following:
\begin{equation}
\label{newman}
\Big|\partial_k \Big( \int_\L \jn(u,v)F(v) dv\Big)\Big|\le \int_\L \jn(u,v)  \big|\partial_k F(v)\big| dv\, ,
\end{equation}
where for $1\le k\le d$, $\partial_k F$ is the partial derivative in the direction $\vecte_k$.
In particular, if $|\cdot|_1$ stands for the $l_1$ norme of $\R^d$, then
\begin{equation}
 |  \nabla(J \star F)(u)\big|_1 \, \le\, \big(J\star \big|\nabla F\big|_1\big) (u)\, .
\end{equation}

\end{lemma}

\smallskip
The proof of this Lemma is similar to the one of Lemma 3.1. in \cite{mo3} and therefore is omitted.

\medskip
Next, we show  that  for $t\ge 0$ and $\mb V =(V_1,
\dots, V_d) \in \big(\mc C^{1,1}( [0,T]\times \L)\big)^d$,
the  rates $C_{N,t}^{\b,V_i}$, $1\le i\le d$  of the generator $ \cl_{\b,N}^{\mb V}$ is a perturbation of the  rate of the
symmetric simple exclusion generator.   
For any $F\in \mc C^1(\L)$, $u\in \L$ and $1\le k\le d$ denote by   $\partial_k^N F (u)$
the discrete (space) derivative in the direction $\vecte_k$: 
\begin{equation}
\label{derive1}
\partial_k^N F (u)=N\big[F (u+e_k/N)-F(u)\big]\, , \ \text{if}\quad u+e_k/N\in\L_N\;.
\end{equation}

\begin{lemma}
\label{b1} Fix $t\ge 0$ and $\mb V =(V_1,
\dots, V_d) \in \big(\mc C^{1,1}( [0,T]\times \L)\big)^d$.
For any $1\le k\le d$, $\eta\in \Sigma_N$
and any $x\in\L_N$ with $x +\vecte_k\in\L_N$,
$$
C_{N,t}^{\b,\mb V_k}(x,x + \vecte_k;\eta)= 1 - N^{-1}\big(\eta(x +\vecte_k)-\eta(x)\big)\Upsilon^{\b,\mb V}_k(\pi^N(\eta),s,x/N)
+O(N^{-2})\, ,
$$
where for $1\le k\le d$, $\eta\in \Sigma_N$, $s\ge 0$ and $x\in\L_N$,
$$
\Upsilon^{\b,\mb V}_k(\pi^N(\eta),s,x/N)\, =\,\b
\partial_k^N( \jn\star \pi^N(\eta_s))(x/N) \,+\, V_k(s,x/N)\, ,
$$

 \end{lemma}

\begin{proof} 
Recall from \eqref{rate-V} that 
$$  C_{N,t}^{\b, V_k} (x,x+\vecte_k;\eta)= 
C_{N}^{\b} (x,x+\vecte_k;\eta) 
e^{-[  \eta(x+e_k)-\eta(x)]N^{-1} V_k(t,x/N)} .
$$
By definition of $H_N$, for all $x,y\in\L_N$ and $\eta\in\Sigma_N$,
\begin{equation*}
\begin{aligned}
&H_N(\eta^{x,y}) - H_N(\eta)=
\frac{2}{N^d}\big(\eta(x)-\eta(y) \big)^2\big( \jn(\frac xN, \frac yN ) - \jn(0,0)  \big) \\
&\ \qquad +
\big(\eta(x)-\eta(y) \big) \frac{2}{N^d}\sum_{z\in\L_N}\eta(z) \big [ \jn(\frac xN, \frac zN ) 
-\jn(\frac yN, \frac zN ) \big]\, .\\ 
\end{aligned}
\end{equation*}
Thus, by Taylor expansion,
$$
C_N^\b(x,x + \vecte_k;\eta)= 1 -\b \big(\eta(x +\vecte_k)-\eta(x) \big)  N^{-1}\partial_k^N\big[ \big( \jn\big)
\star \pi^N(\eta)\big](x/N)+O(N^{-2})\; .
$$

To conclude the proof of the Lemma, it remains to apply again Taylor expansion 
to the exponential function. 
\end{proof}


It is well known that one of the main steps
in the derivation of a large deviations principle for the empirical density is a superexponential estimate
which allows the replacement of local functions by functionals of the empirical density in the large
deviations regime. 
 For a cylinder function $\Psi$ denote the
expectation of $\Psi$ with respect to the Bernoulli product measure
$\nu^N_{\alpha}$ by $\widetilde{\Psi}(\alpha)$:
\begin{equation*}
\widetilde{\Psi}(\alpha) = E^{\nu^N_{\alpha}}[\Psi]\, .
\end{equation*}
For a positive integer $l$ and $x\in\L_N$, denote the empirical
mean density on a box of size $2l+1$ centered at $x$ by $\eta^l(x)$:
\begin{equation*}
\eta^l(x) = \frac{1}{|\Lambda_l(x)|}\sum_{y\in\Lambda_l(x)}\eta(y)\, ,
\end{equation*}
where
\begin{equation}\label{average}
\Lambda_l(x) = \Lambda_{N,l}(x) = \{y\in\L_N:\, |y-x|\leq l\}\, .
\end{equation}
For $1\le j\le d$, define the cylinder function $\Psi_j =[\eta(e_j)-\eta(0)]^2$,
For each $\mb V =(V_1,\cdots,V_d)$, $\mb G=(G_1,\cdots,G_d)$ in $(\mc C^{0,1}([0,T] \times \L))^d$, 
and each $\varepsilon>0$, let
\begin{equation}\label{vne}
\begin{aligned}
&{\mathcal G}_{N,\varepsilon}^{\mb G,\mb V,\b}(s,\eta)\, =\, \frac{1}{N^d}\sum_{j=1}^d\; \sum_{x,x+e_j\in\L_N}
G_j(s,x/N)\\ 
&\qquad\qquad\qquad\qquad
\times \partial_j^N\Upsilon^{\b,\mb V}_j(\pi^N(\eta),s,x/N) \left[\tau_x\Psi_j(\eta)-
{\widetilde \Psi}_j(\eta^{\varepsilon N}(x))\right]\, .
\end{aligned}
\end{equation}

For a continuous function $H:[0,T]\times\Gamma\to\bb R$, let
\begin{equation}\label{vneb}
{\mathcal H}_{N}^H(s,\eta) \, =\, \frac{1}{N^{d-1}} 
\sum_{x\in\Gamma_N} H(s, x/N)
\big[\eta(x)-b(x/N)\big]\, .
\end{equation}

\begin{proposition}
\label{see}
Fix $\mb G ,\mb V\in (\mc C^{0,0}([0,T]\times\L))^d$, $H$ in $\mc
C^{0,0}([0,T]\times\Gamma)$ and $\b \ge 0$. For any sequence of initial measures $\mu_N$ and every $\delta>0$,
\begin{eqnarray*}
&&\limsup_{\varepsilon\to 0}\limsup_{N\to\infty}
\frac{1}{N^d}\, \log \bb{P}_{\mu_N}^{\b,\mb V,H}
\Big[ \, \Big|\int_0^T {\mathcal G}_{N,\varepsilon}^{\mb G,\mb V,\b} (s,\eta_s) \, ds \Big|
>\delta\Big] \;=\; -\infty\, , \\
&&\limsup_{N\to\infty}
\frac{1}{N^d}\, \log \bb{P}_{\mu_N}^{\b,\mb V,H}
\Big[ \, \Big|\int_0^T {\mc H}_{N}^H(s,\eta_s) \, ds \Big|
>\delta\Big] \;=\; -\infty\, .
\end{eqnarray*}
\end{proposition}

The useful tools to derive the superexponential estimate stated in Proposition \ref{see}
is given by the next result
concerning the Dirichlet form $ \langle -\mf L_{N}^{\mb V,H}   \sqrt{f(\eta)},  \sqrt{f(\eta)} \rangle_{\nu^N_{\theta(\cdot)}}$ for the full
dynamics.
For each probability measure $\nu $ on $\Sigma_N$ and each function $f \in L^2 ( \nu)$, define the following  functionals
\begin{equation}
\label{dir-form2}
\begin{aligned}
\cd_{0,N}\big(f, \nu\big) & =\frac12
 \sum_{i=1}^d\sum_{x,x+\vecte_i\in\L_N} \int
  \left( {f}(\eta^{x,x+\vecte_i}) -{f}(\eta) \right)^2 d \nu (\eta) \, ,\\
\cd_{b,N}\big(f,\nu\big) & =\frac12 \sum_{x \in \Gamma_N} \int
r_x\big(b(x/N),\eta \big) 
  \left( {f}(\sigma^{x}\eta) -{f}(\eta) \right)^2 d \nu (\eta)\, .
\end{aligned}
\end{equation}

\begin{lemma}
\label{dirichlet} 
Let $\theta:\overline\L\to (0,1)$ be a smooth function such that
$\theta (\cdot){\big\vert_\Gamma} =\; b (\cdot)$.
There exists two positive 
constants $C_0\equiv C_0(\|\nabla \theta \|_\infty ,J^{\text{neum}}, \mb V)$,  $C_0'\equiv C_0'(b ,H)$ so  that for any   $a >0$  and for 
$f\in L^2\big( \nu^N_{\theta(\cdot)}  \big)$,
 
\begin{equation}\label{dirt1}
\begin{aligned}
 \langle f\, ,\, \cl_{\b,N}^{\mb V}
f\rangle_{\nu^N_{\theta(\cdot)}} &
\le\, -\big(1-{a}\big) {\cd}_{0,N} \big({f},\nu^N_{\theta(\cdot)} \big)  
+\frac{C_0}{a} N^{-2+d}  \| f\|^2_{L^2(\nu^N_{\theta(\cdot)}) },\\
\langle f\, ,\, L_{b,N}^H
f   \rangle_{\nu^N_{\theta(\cdot)}}&
= - \big(1-{a}\big) {\cd}_{b,N} \big({f},\nu^N_{\theta(\cdot)} \big) +\frac{C_0'}{a} N^{-2+d}  
\| f\|^2_{L^2(\nu^N_{\theta(\cdot)}) }\; .
\end{aligned} 
\end{equation}
\end{lemma}

\medskip
The proof of this lemma is similar to the one of Lemma 3.3 in \cite{mo3} and is thus omitted. 

\medskip
We conclude this section by the Girsanov formula needed in the proof of the large deviations. Indeed,
in order to compare the original dynamics to a perturbed dynamics with regular drifts $\mb V, H$ 
\eqref{rate-V} and \eqref{rate-H}, we have to compute the Radon-Nikodym derivative of the modified
process with respect to the original one (see  \cite{kl}, Appendix
1, Proposition 7.3).
Fix a vector-valued function $\mb V\in (\mc C^{0,0}([0,T]\times \L))^d$ and
a function $H\in \mc C^{0,0}([0,T]\times \Gamma)$.
For any initial measure $\mu_N$ and any positive time  $t>0$, the Radon-Nikodym derivative of $\Pb_{\mu_N}^{\b,\mb V,H} $ with respect to $\Pb_{\mu_N}^{\b} $ 
restricted to the time interval $[0,t]$ is gives by

\begin{equation}\label{girsanov}
\frac {d \Pb_{\mu_N}^{\b,\mb V,H} }{d\Pb_{\mu_N}^{\b}}\big((\eta_s)_{s\in [0,t]}\big) \,=\,
\M_t^{\b,\mb V}\, +\, \B_t^{b,H}\, ,
\end{equation}
where $\M_t^{\b,\mb V}$ and $\B_t^{b,H}$ are two exponential martingales given by,
\begin{equation*}
\begin{aligned}
&\M_t^{\b,\mb V} \,=\,
\exp\Big( \sum_{k=1}^d\sum_{x,x+\vecte_k\in \L_N}\Big\{ \int_0^t \, \frac{1}{N} V_k (s,x/N) \, d W_{s}^{x,x+\vecte_k}\\
&\ 
-N^2\int_0^t\big[\eta_s(x)+\eta_s(x+e_j)]
C_{N}^{\b} \big(x,x+\vecte_k; \eta_s\big) \big[ e^{-[\nabla^{x,x+\vecte_k}\eta_s(x)]\frac{1}{N} V_k (s,x/N)}\, -\, 1\big]ds 
\Big\} \Big)\, ,\\
&\B_t^{b,H} \,=\,
\exp\Big(\sum_{x\in \Gamma_N}\Big\{ \int_0^t \, \frac{1}{N} H (s,x/N) \, d W_{s}^{x}\\
&\ \ \ \ 
-N^2\int_0^t r_x\big(b(x/N),\eta_s(x)\big)
\big[ e^{[2\eta_s(x)-1]\frac{1}{N} H (s,x/N)}\, -\, 1\big]ds 
\Big\} \Big)\, ,
\end{aligned}
\end{equation*}
where the rate $r_x(\cdot,\cdot)$ is given by \eqref{rate-b} and for any 
function $g:\Sigma_N\to \R$ and $x,y\in\L_N$, we have denoted $\nabla^{x,y}g(\eta)=[g(\eta^{x,y})-g(\eta)]$.

\medskip
\section{Properties of the rate functionals}\label{secprf}
In this section, we prove representation results for the rates ${\mc J}_T^\g(\cdot)$ and $\mc I_T^\g(\cdot)$, see 
Lemma \ref{rep0}, the lower semicontinuity and the compactness of the level sets , see Proposition \ref{s03}. 

\subsection{Lower semicontinuity} 
We first prove that the functional ${\mc J}_T^\g$ is larger than $\mc I_T^\g$:

\begin{lemma}\label{compar1} For any $(\mb W,\pi)\in D([0,T],{\mc M}^{d+1})$,
$$
{\mc I}_T^\g(\pi)\, \le\, {\mc J}_T^\g(\mb W,\pi)\, .
$$ 
\end{lemma}

\begin{proof} 
When ${\mc J}_T^\g(\mb W,\pi)=+\infty$, the inequality is
trivially verified. Suppose then that ${\mc J}_T^\g(\mb W,\pi)<+\infty$. 
This implies that $\pi\in D\big([0,T],\mc F^1\big)$, $(\mb W,\pi)\in \mf A_\g$, $\mc Q(\pi)<+\infty$ and
${\mc J}_T^\g(\mb W,\pi)=\J_T (\mb W,\pi)$. Furthermore, by definition, since 
$\pi\in D\big([0,T],\mc F^1\big)$ and $\mc Q(\pi)<+\infty$, we have
$\mc I_T^\g(\mu) =\I_T^\g(\mu)$.

Let $F\in \mc C_0^{1,2}([0,T]\times \L)$, since $(\mb W,\pi)\in \mf A_\g$, we have
$$
\widehat\I_F^{T,\g}(\pi)\, =\, {\widehat \J}^T_{\nabla F}(\mb W,\pi)\, \le\, {\mc J}_T^\g  (\mb W,\pi)\, .
$$
To conclude the proof, it is enough to take the supremum over all $F\in \mc C_0^{1,2}([0,T]\times \L)$,
on the left hand side of the last inequality.
\end{proof}

The main result of this subsection is stated in the following proposition.

\begin{proposition}
\label{s03}
For every profile $\gamma\in\mc F^1$, the functional
${\mc J}_T^\g $, resp. ${\mc I}_T^\g$ defined in \eqref{f10}, resp. \eqref{3:Ib}
is lower semicontinuous for the topology of the space
$D([0,T],\mc M^{d+1})$, resp. $D([0,T],\mc M)$. Moreoever the functional ${\mc I}_T^\g$ is of
compact level sets in $D([0,T],\mc M)$.
\end{proposition}

The proof is split in several lemmata. We follow the general scheme used in  
\cite{qrv, blm}.
Denote 
\begin{equation*}
\begin{aligned}
&\cb_\gamma^{b(\cdot)} = \{ (\pi_t(du))_{t\in[0,T]}=(\rho_t(u)du)_{t\in[0,T]}\ :\ 
\rho\in L^2([0,T], H^1(\L)) \, ,\\ 
&\qquad\qquad\qquad\qquad\qquad
\rho_0(\cdot) = \gamma (\cdot); \quad \text{\bf Tr}(\rho_t)(\cdot) = b(\cdot), 
 \;\;  \text{for}\ \text{a.e.}\ t\in (0,T]\}. 
 \end{aligned}
\end{equation*}

\begin{lemma}
\label{lem01}
Let $\pi$ be a trajectory in $D([0,T],\mc M)$ such that
${\mc I}_T^\g(\pi)<\infty$. Then $\pi$ belongs to $\cb_\gamma^{b(\cdot)}\cap C([0,T], \mc F^1)$.
Furthermore, there exists a positive constant $C_0=C_0(\b,\jn)$ such that 
\begin{equation}\label{lem01-0}
\mc Q (\pi) \;\le\;  C_0\big\{ 1 + {\mc I}_T^\g (\pi)\big\}\, .
\end{equation}
\end{lemma}
\begin{proof}
The proof of the first statement of this Lemma  is similar to the one of Lemma 4.1 in \cite{flm} and is therefore omitted.
One can prove \eqref{lem01-0} by using the same arguments as in the proof of 
Proposition 4.3. \cite{qrv} or Lemma 4.9. in \cite{blm}.
\end{proof}

The proof of the lower-semicontinuity of the rate function $\mc I_T^\g$ is based on compactness arguments; its
basic tools is given by the next Proposition. We refer to \cite{blm, flm}  for the proof.

\begin{proposition}
\label{g06p} 
Let $\{\pi^n : n\ge 1\}$ be a sequence of functions in $D([0,T],\mc M)$ such that
$$
\sup_{n\in\N} \big\{\mc I^\g_T (\pi^n)\big\}< \infty\, 
$$
with $\pi^n(t,du)=\rho^n(t,u)du$, for $t\in[0,T]$ and $n\in\N$.
Suppose that the
sequence $\rho^n$ converges weakly in $L^2 ([0,T]\times \L)$ to some $\rho$.
Then, $\rho^n$ converges strongly in $L^2 ([0,T]\times \L)$ to $\rho$.
\end{proposition}

\noindent{\it Proof of Proposition \ref{s03}.}
The proof for the functional ${\mc I}_T^\g$ is omitted since it's the same as for
the one dimensional boundary driven Kawasaki process with Neuman Kac interaction  \cite{mo3}.

To prove the lower semicontinuity of the functional ${\mc J}_T^\g$, we have to show that for all
$a \ge 0$ the set
\begin{equation*}
\mb E_a \; =\; \Big\{(\mb W , \pi) \in D([0,T],\mc M^{d+1}) \ : \ 
{\mc J}_T^\g (\mb W,\pi)\le a \Big\} 
\end{equation*}
is closed in $D([0,T],\mc M^{d+1})$.  
Fix $a\ge 0$ and consider a
sequence $\{(\mb W^n,\pi^n) : n\ge 1\}$ in $\mb E_a$ converging to some $(\mb W,\pi)$
in $D([0,T],\mc M^{d+1})$, and denote by $\pi^n_t(du)=\rho^n_t(u)du$.
Then for all $\mb V$ in $(\mc C ([0,T]\times \L))^d$ and $F$ in $\mc C ([0,T]\times \L)$,
\begin{equation}\label{gf06}
\begin{aligned}
&\lim_{n \to \infty} \int_0^T dt\, \< \mb W^n_t ,\mb V_t \>
= \int_0^T dt\, \<\mb W_t ,\mb V_t \> \, ,  \\ 
&\lim_{n \to \infty} \int_0^T dt\, \< \pi^n_t ,F_t \>
=  \int_0^T dt\, \<\pi_t ,F_t \> . 
\end{aligned}
\end{equation}

We claim that $\mc E^\g(\mb W,\pi)< +\infty$. Indeed, from the lower semicontinuity of ${\mc I}_T^\g$, 
Lemma \ref{compar1} and Lemma \ref{lem01}, $\pi$ belongs to $\mc B_\g^{b(\cdot)}$ and
$\mc Q (\pi) \;\le\;  C_a$
for some positive constant $C_a$. 
Moreover, for any $F\in \mc C^1_0(\L)$
\begin{equation*}
\begin{aligned}
0& \; =\lim_{n\to \infty}\sup_{t\in[0,T]}\Big\{\langle \pi_t^n , F\rangle - \langle \gamma , F\rangle -
\langle \mb W_t^n , \nabla F \rangle\Big\}\\
\ &\ \  =\sup_{t\in[0,T]}\Big\{ \langle \pi_t , F\rangle - \langle \gamma , F\rangle -
\langle \mb W_t , \nabla F \rangle \Big\} \, ,
\end{aligned}
\end{equation*}
proving that $(\mb W,\pi)\in \mf A_\g$ and then $\mc E^\g(\mb W,\pi)< +\infty$, so that
${\mc J}_T^\g(\mb W,\pi)= \J_T (\mb W,\pi)$.

Denote by $\rho$ the density of $\pi$: $\pi_t(du)=\rho_t(u)du$.
Since $\rho^n$ converges weakly to $\rho$ in $L^2([0,T]\times \L)$ (cf. \eqref{gf06}), by Lemma \ref{g06p}, 
$\rho_n$ converges strongly to
$\rho$ in $L^2 ([0,T]\times \L)$, hence for any $\mb V$ in $(\mc C^{1,1} ([0,T]\times \L))^d$
\begin{equation*}
\begin{aligned}
 &\lim_{n\to\infty}\Big\{ \mb L_{\mb V}^\b (\mb W^n, \pi^n)
\;-\; \frac 12 \int_0^T dt \, \< \sigma(\rho_t^n),\mb V_t\cdot  \mb V_t \>\Big\}\\
&\qquad\qquad\qquad\qquad\qquad
=\mb L_{\mb V}^\b (\mb W, \pi)
\;-\; \frac 12 \int_0^T dt \, \< \sigma(\rho_t),\mb V_t\cdot  \mb V_t \>\, .
\end{aligned}
\end{equation*}
Since $(\mb W^n,\pi^n)$ belongs to $\mb E_a$, the left hand side is bounded by
$a$. Taking the supremum over $\mb V$ in $(\mc C^{1,1} ([0,T]\times \L))^d$
we obtain that $\J_T (\mb W,\pi)\le a$ and conclude the proof of 
the lower semicontinuity of ${\mc J}_T^\g$.
\cqfd

\subsection{Representation theorem}
Given a path $\pi\in D([0,T];\mc F^1)$ with $\pi(t,du)=\rho(t,u)du$,
we denote by $\LL^2(\s(\pi))$
the Hilbert space of vector-valued functions 
$\mb G:[0,T]\times \L
\to \bb R^d$ endowed with the inner product $\<\!\<\cdot, \cdot \>\!\>_{\s(\pi)}$ defined by
\begin{equation*}
\<\!\< \mb V, \mb G\>\!\>_{\s(\pi)} \;=\; \int_0^T dt \int_{\L} du\,
\s(\pi(t,u)) \, \mb V(t,u) \cdot \mb G(t,u)\;.
\end{equation*}
The norm of $\LL^2(\sigma(\pi))$ is denoted by $\Vert\cdot\Vert_{\LL^2(\sigma(\pi))}$.

Denote by $H^1_0(\sigma(\pi))$ the Hilbert space
induced by $C^{1,2}_0([0,T]\times \L)$ endowed with the inner
product $\langle\cdot,\cdot\rangle_{1,\sigma(\pi)}$ defined by
\begin{equation*}
\langle F,H\rangle_{1,\sigma(\pi)}
=\<\!\< \nabla F, \nabla H\>\!\>_{\s(\pi)}\, .
\end{equation*}
Induced means that we first declare two functions $F,H$ in $\mc C^{1,2}_0([0,T]\times \L)$ to be equivalent if $\langle
F-H,F-H\rangle_{1,\sigma(\pi)} = 0$ and then we complete the quotient
space with respect to the inner product
$\langle\cdot,\cdot\rangle_{1,\sigma(\pi)}$. 
The norm of $H^1_0(\sigma(\pi))$ is denoted by $\Vert\cdot\Vert_{H^1_0(\sigma(\pi))}$.

\begin{lemma}\label{rep0} Let $(\mb W,\pi)\in D([0,T],\mc M^{d+1})$ such that ${\mc J}_T^\g(\mb W,\pi) <\infty$.
There exists a function $\mb U$ in $\LL^2(\sigma(\pi))$ so that $\mc E^\g (\mb W, \pi)<\infty$ and
$(\mb W, \pi)$ is the weak solution of the equation
\begin{equation}\label{rep2}
\partial_t \mb W_t  \;=\; 
-\nabla \rho_t +\sigma(\rho_t) \big[\b \nabla (\jn  \star \rho_t )  \, +\, \mb U\big]\, ,
\quad \mb W_0 =0\, ,
\end{equation}
in the following sense : for any ${\mb G}\in \big(C^{1,1}([0,T]\times \L) \big)^d$,
\begin{equation*}\label{rep2-0}
\mb L_{\mb G}^\b (\mb W, \pi) \;=\; \<\!\< \mb G, \mb U\>\!\>_{\s(\pi)}
\;=\; \int_0^T dt\, \< \s(\pi_t) , \mb G_t \cdot \mb U_t\>\, ,
\end{equation*}
where the linear function $\mb G \mapsto \mb L_{\mb G}^\b (\mb W, \pi)$ is defined by \eqref{eq:4.6}.

Furthermore, there exists a function $F\in H_0^{1}(\s(\pi))$ such that $\rho(\cdot,\cdot)$ solves
the equation \eqref{eq:V} and 
$\text{\bf div} \big(\s(\rho)(\mb U-\nabla F)\big)=0$ in the weak sens described by \eqref{rep3a}.
Moreover,
\begin{equation}\label{rep1}
 {\mc J}_T^\g (\mb W, \pi) \;=\frac12\; \Vert\mb U\Vert_{\mb L^2(\sigma(\pi))}
 \;=\;\frac 12 \int_0^T dt \, \< \s(\rho_t),\mb U_t\cdot  \mb U_t \>
\end{equation}
and 
\begin{equation}
\label{rep3}
\mc I_T^\g(\pi)\, =\, \frac12\; \Vert F\Vert_{H^1_0(\sigma(\pi))}
\, =\, \frac12 \int_0^Tdt\; \langle\sigma(\rho_t),\nabla F_t\cdot\nabla F_t\rangle\, .
\end{equation}
Here {\bf div} stands for the divergence operator.
\end{lemma}

\begin{proof}
Assume that ${\mc J}_T^\g(\mb W,\pi) <\infty$, then $\mc E^\g (\mb W, \pi)<\infty$ and $\J_T (\mb W,\pi) <\infty$.
Following the arguments in \cite[\S 10.5]{kl}, from Riesz representation theorem, we derive the existence
of a function $\mb U$ in $\LL^2(\sigma(\pi))$ satisfying \eqref{rep1} and \eqref{rep2}.

On the other hand, from Lemma \ref{compar1}, we have $\mc I_T^\g(\pi) < \infty$. 
Using again the Riesz representation theorem (cf. \cite[\S 10.5]{kl} ), we derive the existence
of a function $F$ in $H^1_0(\sigma(\pi))$ such that $\rho$ is the  weak solution of  the boundary value  problem \eqref{eq:V},
with $\mb V=\nabla F$. Then, the representation \eqref{rep3} for the functional $\mc I_T^\g$ follows immediately. 
Finally, equation \eqref{rep2} and the fact that $(\mb W,\pi)\in \mf A_\g$  yields,
\begin{equation}\label{rep3a}
\big\<\!\<(\mb U-\nabla F),\nabla G\big\>\!\>_{\s(\rho)}\, =\, 0, 
\end{equation}
for all $G\in C_0^{1,2}([0,T]\times \L) $.
\end{proof}

\medskip
\section{large deviations upper bound for the empirical current}\label{secldub}
In this section, we prove the large deviations upper bounds stated in Theorem \ref{gd-(d,c)} and in Theorem \ref{s02}. 
In view of the definitions of the energy functional ${\mc E}^\gamma$ and the rate functional for the large deviations,
we need to exclude in the large deviation regime, paths $(\mb W_t,\pi_t)_{t\in[0,T]}$ which do not belong to 
$\mf A_\gamma $, 
and with infinite energy $\mc Q(\pi)=+\infty$. 

\subsection{The set $\mf A_\g$} Fix a positive profile $\g$ and let $\widetilde {\mf A}_\g$ be the set of trajectories
$(\mb W, \pi)$ in $D([0,T], \mc M^{d+1})$ such that 
for any $G \in \mc C_0^2(\L)$ and any $\varphi\in \mc C^1([0,T])$ 
$$
\sup_{0\le t\le T}\V_{(G,\varphi)}^{t,\g}(\mb W,\pi)\, =\, 0\, ,
$$
where for $(G,\varphi)\in \mc C_0^2(\L)\times  \mc C^1([0,T])$ and $0\le t\le T$,
\begin{equation}\label{eq:4.5b}
\begin{aligned}
\V_{(G,\varphi)}^{t,\g}(\mb W,\pi) \, =&\langle \pi_t , G\rangle\varphi(t) - \langle \gamma , G\rangle \varphi(0)
-\int_0^t ds \langle \pi_s , G\rangle  \varphi'(s)\\
& \ \ -
\langle \mb W_t , \nabla G \rangle \varphi(t) + 
\int_0^t ds \langle \mb W_s , \nabla G\rangle \varphi'(s)\, .
\end{aligned}
\end{equation}
Here $\varphi'$ stands for the time derivative of $\varphi$.

\begin{lemma}\label{t-u} Fix $(\mb W,\pi)$  in $D([0,T],\mc M^{d+1})$ such that
$$
\sup_{(G,\varphi)}\sup_{0\le t\le T} \big\{\V_{(G,\varphi)}^{t,\g}(\mb W,\pi) \big\}< \infty\, ,
$$
where the supremum is taken over all $(G,\varphi)\in \mc C_0^2(\L)\times  \mc C^1([0,T])$.
Then $(\mb W,\pi)$ belongs to $\mf A_\g$.
\end{lemma}

\begin{proof}
Let $M>0$ be such that $\V_{(G,\varphi)}^{t,\g}(\mb W,\pi)\le M$,
for all $(G,\varphi)\in \mc C_0^2(\L)\times  \mc C^1([0,T])$, and $0\le t\le T$. 
Fix a function $G\in \mc C_0^2(\L)$ and $0\le t_1<t_2\le T$, we have
\begin{equation*}
\big\{\<\pi_{t_1}, G\> - \<\pi_{t_2} , G\>\big\} -\big\{\big\<{\mb W}_{t_1}, \nabla G\big\> - 
\big\<{\mb W}_{t_2} , \nabla G\big\>\big\}
\,\le  M\, \;.
\end{equation*} 
Applying this last inequality to the functions  $- G$ and then to $AG$ for positive number $A > 0$, we get,
$$
\Big|\big\{\<\pi_t, G\> - \<\pi_s , G\>\big\} -\big\{\big\<{\mb W}_t, \nabla G\big\> - \big\<{\mb W}_s ,\nabla G\big\>\big\}\Big|
\le \frac{M}{A}\, ,
$$
for all $A>0$. It remains to let $A\uparrow +\infty$.
\end{proof}

The following lemma allows to prove that we may set
the large deviations rate functional equal to $+\infty$ on the set of paths $(W,\pi)$, which do not belong to 
$\widetilde {\mf A}_\g$.

\begin{lemma}\label{t-u2}
Fix a sequence $\{\eta^N \in \Sigma_N : N\ge 1\}$ of configurations.
For any $(G,\varphi)\in \mc C_0^2(\L)\times  \mc C^1([0,T])$ and
any  $a >0$,  we have
\begin{equation*}
 \limsup_{N\to\infty} \frac 1{N^d}  \log
\bb E^{\b}_{\eta^N} \Big [ 
\exp\Big(a \, N^d \sup_{0\le t\le T} \V_{(G,\varphi)}^{t,\g}({\mb W}^N,\pi^N)\Big) \Big] 
\; \le \; 0\;.
\end{equation*}
\end{lemma}

\begin{proof} The proof follows the general scheme used in \cite{bdgjl2}.
Notice however that in our context there are some additional difficulties due to the boundary terms.
Fix $(G,\varphi)\in \mc C_0^2(\L)\times  \mc C^1([0,T])$. For any 
time $s\in[0,T]$, we have the following microscopic relation
\begin{equation*} 
\eta_s(x) = \eta_0(x)+\sum_{j=2}^d\big( W_s^{x-e_j,x} - W_s^{x,x+e_j}\big)+
\begin{cases}
\displaystyle
 W_s^{x-e_1,x} - W_s^{x,x+e_1} \; \textrm{if} \ x\in \L_N\setminus\Gamma_N\, ,\\
\displaystyle
-W_s^{x,x+e_1}-W_s^x \quad \textrm{if} \ x\in \Gamma_N^-\, ,\\
\displaystyle
W_s^{x-e_1,x}-W_s^x \quad \textrm{if} \ x\in \Gamma_N^+\, .\\
\end{cases}
\end{equation*}

Since $G$ vanishes at the boundary $\Gamma$,
the classical spatial summations by parts and integrations by parts in time, permit to rewrite the two terms of 
$\V_{(G,\varphi)}^{t,\g}({\mb W}^N,\pi^N)$  as
\begin{equation*}
 \begin{aligned}
&\langle \pi_t , G\rangle\varphi(t) - \langle \pi_0 , G\rangle \varphi(0)
-\int_0^t ds \langle \pi_s , G\rangle  \varphi'(s)\\
&\qquad\qquad  \ \ \  \, =\, \frac{1}{N^{d+1}}\sum_{j=1}^d\sum_{x\in\L_N\setminus\Gamma_N^+} \int_{0}^t \partial_j^N G(x/N) \varphi(s) dW_s^{x,x+e_j} \, ,\\
&\langle \mb W_t , \nabla G \rangle \varphi(t) -
\int_0^t ds \langle \mb W_s , \nabla G\rangle \varphi'(s)\\
& \qquad \qquad \ \ \   \, =\, \frac{1}{N^{d+1}}
\sum_{j=1}^d\sum_{x\in\L_N\setminus\Gamma_N^+} \int_{0}^t \partial_{j} G(x/N) \varphi(s) dW_s^{x,x+e_j}\\
& \qquad \qquad \qquad \ \   \, +\,  \frac{1}{N^{d+1}}
\sum_{x\in\Gamma_N} \int_{0}^t \partial_{1} G(x/N) \varphi(s) dW_s^{x} \, ,
 \end{aligned}
\end{equation*}
where $\partial_{j}^N G(x/N)$ is the discrete derivative defined in \eqref{derive1} and $\partial_{j} G$ is the partial
derivative of the function $G$ in the direction $e_j$.
Let $H$ be the function given by $H(s,u)=-\partial_1 G(u)\varphi(s)$ and 
for $1\le j\le d$ and $N>1$, denote by $\mb V^N=(V_1^N,\cdots, V_d^N)$ the time dependent vector valued function defined by
$\mb V_j^N(s,u)= N\big[\partial_j^N G(u) -\partial_{j} G(u)\big]\varphi(s)$, we obtain
\begin{equation*}
\begin{aligned}
aN^d\V_{(G,\varphi)}^{t,\g}({\mb W}^N,\pi^N) & =\frac{a}{N^{2}}
\sum_{j=1}^d\sum_{x\in\L_N\setminus\Gamma_N^+} \int_{0}^t V_j^N(s,x/N)  dW_s^{x,x+e_j}\\
\ &\ \ \ \   +\frac{a}{N}
\sum_{x\in\Gamma_N} \int_{0}^t H(s,x/N) dW_s^{x} \, .
\end{aligned}
\end{equation*}
Thus by Cauchy-Schwarz inequality,
\begin{equation}\label{t-u4b}
\begin{aligned}
&\frac 1{N^d}  \log
\bb E^{\b}_{\eta^N} \Big [ 
\exp\Big(aN^d \sup_{0\le t\le T} \V_{(G,\varphi)}^{t,\g}({\mb W}^N,\pi^N)\Big) \Big] \\
&\ \; \le \;\frac 1{2N^d}  \log
\bb E^{\b}_{\eta^N} \Big [ 
\exp\Big(\frac{2a}{N^{2}}\sup_{0\le t\le T}
\sum_{j=1}^d\sum_{x\in\L_N\setminus\Gamma_N^+} \int_{0}^t V_j^N(s,x/N)  dW_s^{x,x+e_j}\Big) \Big] \\
&\ \ \ +  \;\frac 1{2N^d}  \log
\bb E^{\b}_{\eta^N} \Big [ 
\exp\Big(\frac{2a}{N}\sup_{0\le t\le T}
\sum_{x\in\Gamma_N} \int_{0}^t H(s,x/N) dW_s^{x}\Big) \Big] \,.
\end{aligned}
\end{equation}
Next, we control separately the two terms of the right hand side of \eqref{t-u4b} using 
the mean one exponential martingales ${\bb M}_t^{\b, \frac{2a}{N} {\mb V}}$ and ${\bb B}_t^{b, 2aH}$ defined
in the Girsanov formula \eqref{girsanov}: 
\begin{equation}\label{M-B}
\begin{aligned}
{\bb M}_t^{\b, \frac{2a}{N}{\mb V}} & \, =\, 
\exp\Big(\frac{2a}{N^{2}}
\sum_{j=1}^d\sum_{x\in\L_N\setminus\Gamma_N^+} \int_{0}^t V_j^N(s,x/N)  dW_s^{x,x+e_j}
\, -\, R_{0,t}^{\frac{2a}{N}{\mb V}}\Big)\, ,\\
{\bb B}_t^{b, 2aH} &\, =\, 
\exp\Big(\frac{2a}{N}
\sum_{x\in\Gamma_N} \int_{0}^t H(s,x/N) dW_s^{x} \, -\,  R_{b,t}^{2aH}\Big)\, ,
\end{aligned}
\end{equation}
where
\begin{equation*}
\begin{aligned}
R_{0,t}^{\frac{2a}{N}\mb V} &\, =\,
N^2\sum_{j=1}^d\sum_{x\in\L_N\setminus\Gamma_N^+}\int_0^t 
\Big\{[\eta_s(x)+ \eta_s(x+e_j)] C_{N}^{\b} \big(x,x+\vecte_j; \eta_s\big) \times \\
&\qquad\qquad\qquad\qquad\qquad\qquad\qquad\qquad
\Big( e^{-[\nabla^{x,x+\vecte_j}\eta_s(x)]\frac{2a}{N^2} V_j^N (s,x/N)}\, -\, 1\Big) 
\Big\} ds\\
R_{b,t}^{2aH} &\, = \, N^2\sum_{x\in\Gamma_N}\int_0^tr_x\big(b(x/N),\eta_s(x) \big)
\Big( e^{[2\eta_s(x)-1]\frac{2a}{N} H(s,x/N)}\, -\, 1\Big)ds \, .
\end{aligned}
\end{equation*}

We start by the boundary term which differs from the proof of \cite{bdgjl2}.
Recall from \eqref{vneb} the definition of ${\mathcal H}_{N}^H(s,\eta)$. Let $\delta >0$, and define the set
$$
E_{N,\delta}^{H}\,=\, \Big\{ \eta_\cdot \in D([0,T], {\Sigma}_N) : 
\Big| \int_0^T {\mc H}_N^{H}(t,\eta_t) dt \Big| \le \delta\Big\} \, .
$$
According to the definition of ${\bb B}_t^{b, 2aH}$ and
using inequality \eqref{maxlim}, we reduce the control of the second term of the right hand side of  
\eqref{t-u4b} to the following claims. 
For any $\delta>0$,
\begin{equation}\label{t-u6}
\limsup_{N\to \infty}\frac 1{2N^d}  \log
\bb E^{\b}_{\eta^N} \Big [ 
\exp\Big(\sup_{0\le t\le T}
\Big\{{\bb B}_t^{b, 2aH}\times \exp\big(R_{b,t}^{2aH}\big)\Big\} 1_{(E_{N,\delta}^{H})^c}\Big] = -\infty\,.
\end{equation}
and
\begin{equation}\label{t-u6b}
\limsup_{\delta\to 0}\limsup_{N\to \infty}\frac 1{2N^d}  \log
\bb E^{\b}_{\eta^N} \Big [ 
\exp\Big(\sup_{0\le t\le T}
\Big\{{\bb B}_t^{b, 2aH}\times \exp\big(R_{b,t}^{2aH}\big)\Big\} 1_{E_{N,\delta}^{H}}\Big] \le 0\,.
\end{equation}

By Schwartz inequality, the expression in the first limit 
is bounded above by
\begin{equation*}
\limsup_{N\to \infty}\frac 1{4N^d}  \log
\bb E^{\b}_{\eta^N} \Big [ \sup_{0\le t\le T}
\Big({\bb B}_t^{b, 2aH}\times \exp\big(R_{b,t}^{2aH}  \big)\Big)^2\Big]
\, +\,\limsup_{N\to \infty}\frac 1{4N^d}  \log
\bb P^{\b}_{\eta^N} \Big [(E_{N,\delta}^{H})^c\Big]\, .  
\end{equation*}
From Lemma \ref{see}, for any $\delta >0$, the second term in the last expression is equal to $-\infty$. 
Consider the first term.
Since $G\in \mc C_0^2(\L)$, a Taylor expansion shows that
$\displaystyle \sup_{0\le t\le T}|R_{b,t}^{2aH}|\le a(1+\frac{a}{N})C(H,T)$ for some constant $C(H,T)$
depending on $H$ and $T$. Moreover, we can write the martingale ${\bb B}_t^{b, 2aH}$, as 
\begin{equation}\label{t-u52}
\begin{aligned}
\bb B_t^{b, {2a}H } &\, =\,
\left(\bb B_t^{b, aH}\right)^{2}
\exp\Big(2R_{b,t}^{aH} - R_{b,t}^{2aH}\Big)\\
\ & \ \  \le \left(\bb B_t^{b, aH}\right)^{2} \exp\big\{a(1+\frac{a}{N})C(H,T)\big\}\, .
\end{aligned}
\end{equation}
Here and below $C(H,T)$  is a bounded constant depending on $H$ and $T$ whose value may change from line to line.
Therefore,
\begin{equation*}
 \begin{aligned}
 & \frac 1{4N^d}  \log
\bb E^{\b}_{\eta^N} \Big [ \sup_{0\le t\le T}
\Big({\bb B}_t^{b, 2aH}\times \exp\big(R_{b,t}^{2aH}  \big)\Big)^2\Big]\\
&\quad \le 
  a(1+\frac{a}{N})C(H,T)\, +\, 
  \frac 1{4N^d}  \log
\bb E^{\b}_{\eta^N} \Big [ \sup_{0\le t\le T}
\Big({\bb B}_t^{b, 2aH}\Big)^2\Big]\, .
 \end{aligned}
\end{equation*}
Since $\big({\bb B}_t^{b, 2aH}\big)_{t\in [0,T]}$ is a positive martingale equal to $1$ at time $0$, by
Doob's inequality (cf. Proposition 2.16. in \cite{ek}), the last expression in bounded above by
\begin{equation}\label{t-u5b2}
\begin{aligned}
&a(1+\frac{a}{N})C(H,T)\, +\, \frac 1{4N^d}  \log  
\bb E^{\b}_{\eta^N} \Big [4 \Big( {\bb B}_T^{b, 2aH}\Big)^2\Big]\\
&\ \ \ \le a(1+\frac{a}{N})C(H,T)\, +\, \frac 1{4N^d}  \log  
\bb E^{\b}_{\eta^N} \Big[  {\bb B}_T^{b, 4aH} \Big]
\, = \,a(1+\frac{a}{N})C(H,T)\, ,
\end{aligned}
\end{equation}
where we have used again the identity \eqref{t-u52}. 
This concludes the proof of \eqref{t-u6}.

On the other hand, a Taylor expansion shows that on the set $E_{N,\delta}^{H}$, for any $0\le t\le T$, we have
$$
|R_{b,t}^{2aH} | \le N^da\big( \delta +\frac{a}N C(H)\big)\, ,
$$
for some positive constant $C(H)$. We then check the limit \eqref{t-u6b}
by using again the same arguments as in \eqref{t-u52}, \eqref{t-u5b2} and letting $N\uparrow \infty$ then $\d\downarrow 0$.

We now consider the first term of the right hand side of \eqref{t-u4b}. 
Since $G\in \mc C_0^2(\L)$, Lemma \ref{b1}, a Taylor expansion and a
summation by parts allow to show that for any $0\le t\le T$,
\begin{eqnarray*}
\begin{aligned}
R_{0,t}^{\b, \frac{2a}{N}\mb V}&\le 
a o_{\mb V} (1)  \sum_{j=1}^d\sum_{x\in\L_N} \int_0^T dt\, \eta_t(x) +{a\b T}N^{d-1}C(\mb V)
+t{a^2}N^{d-2}C(\mb V,\b) \\
&\ \ \le \, a\Big\{ o_{\mb V} (1) + \frac{\b}N C(\mb V) +\frac{a}{N^2}C(\mb V,\b)  \Big\}  N^d T\, ,
\end{aligned}
\end{eqnarray*}
where $o_{\mb V} (1)$ is an expression depending on $\mb V$ which
vanishes as $N\uparrow\infty$.
It remains to apply again the same arguments as in \eqref{t-u52}, \eqref{t-u5b2} for the martingale 
$\bb M_t^{\b, \frac{2a}{N}{\mb V} }$:
\begin{equation*}
\begin{aligned}
\bb M_t^{\b, \frac{2a}{N}{\mb V} } &\, =\,
\left(\bb M_t^{\b, \frac{a}{N}{\mb V}}\right)^{2}
\exp\Big(2R_t^{\b, \frac{a}{N}\mb V} - R_t^{\b, \frac{2a}{N}\mb V}\Big)\\
& \ \  \le \left(\bb M_t^{\b, \frac{a}{N}\mb V}\right)^{2} e^{ N^d r_N (\mb V,a,T)}\, ,
\end{aligned}
\end{equation*}
where $r_N{(\mb V,a,T)}$  stands for an expression depending on $\mb V, a$ and $T$ which vanishes as $N\uparrow\infty$.
\end{proof}

\medskip
\subsection{The energy estimate $\mc Q$.} In this subsection, we state an energy estimate which is one of the main ingredients 
in the proof of large deviations and also in the proof of hydrodynamic limit.
 For $\pi\in D\big( [0,T], \mc F^1\big)$, with $\pi_t(du)=\rho_t(u)du$, $0\le t\le T$,
$\delta >0$, $1\le i\le d$, and $H\in \mc C^\infty_c ([0,T]\times \L)$ 
define 
 \begin{equation}\label{ccc}
\tq^\d_{i,H}(\pi)\;=\;
\int_0^T dt\<\pi_t, \partial_i H_t \> -\d \int_0^T dt\<\sigma(\rho_t) H_t, H_t \>  \,,
\end{equation}
\begin{equation*}\label{cccb}
\tq^\d_i (\pi)\; =\;  \sup_{H\in C^\infty_c([0,T]\times \L)} \Big\{\tq^\d_{i,H}(\pi)\Big\}\; ,
\end{equation*}

Notice that
$$
{\mc Q} (\pi)\, =\, \frac{\d}2\sum_{i=1}^d
\tq^\d_i (\pi)\,  ,
$$
where  ${\mc Q} (\cdot)$   is defined   in \eqref {1:Q}. We shall denote  ${\mc Q}_i =\tq_i^{2}$,
so that ${\mc Q} =\sum_{i=1}^d {\mc Q}_i$.

 For each $\varepsilon> 0$ and $\pi$ in $\mc M$, denote by
$\Xi_\varepsilon (\pi) = \pi^{\varepsilon}$ the absolutely continuous
measure obtained by smoothing the measure $\pi$:
\begin{equation*}
\Xi_\varepsilon (\pi) (du) \;=\; \pi^{\varepsilon} (du) \;=\; 
\frac 1{\kappa_\varepsilon} \frac {\pi(\bs \Lambda_\varepsilon(u))}
{|\bs \Lambda_\varepsilon(u)|} \,\, du\;,
\end{equation*}
where $\bs \Lambda_\varepsilon(x)$  is defined in \eqref{average}, $|A|$ stands for the Lebesgue measure of the set $A$,
and $\{\kappa_\varepsilon : \varepsilon >0\}$ is a strictly decreasing
sequence converging to $1$. 
Denote 
\begin{equation*}
\pi^{N,\varepsilon} \;=\; \Xi_\varepsilon (\pi^N)\;,
\end{equation*}
and notice that for $N$ sufficiently large
$\pi^{N,\varepsilon}$ belongs to $\mc F^1$ 
because $\kappa_\varepsilon >1$. Moreover, for
any $G\in \mc C^0 (\L)$,
\begin{equation*}
\<\pi^{N,\varepsilon}, G\> \;=\; \frac 1{N^d} \sum_{x\in
  \L_N} G(x/N) \eta^{\varepsilon N}(x) \; +\; O(N, \varepsilon)\;,
\end{equation*}
where $O(N, \varepsilon)$ is absolutely bounded by $C \{ N^{-1} +
\varepsilon\}$ for some finite constant $C$ depending only on $G$.

\begin{lemma}
\label{hs02}  Fix a sequence $\{\eta^N \in \Sigma_N : N\ge 1\}$ of configurations and $H\in\mc C^{\infty}_c([0,T]\times \L)$.
There exists a positive constant $C_1$  depending only on $b(\cdot)$ and $\b$ so that   for any given  $\d_0 >0$,  
for  any $\delta$,  $0\le \delta\le \delta_0$ and any $1\le i\le d$, we have  
\begin{equation*}
\limsup_{\epsilon \to 0} \limsup_{N\to\infty} \frac 1{N^d}  \log
\bb E^{\b}_{\eta^N} \Big [ 
\exp\Big(\delta \, N^d \tq_{i,H}^{\d_0}(\pi^{N,\epsilon}\big)\Big) \Big] 
\; \le \; C_1( T+1)\;.
\end{equation*}
\end{lemma}

The proof of this Lemma is similar to the one of Lemma 3.8. in \cite{mo3}, and therefore is omitted.

\begin{corollary}
\label{hs03}  
Fix a sequence $\{\eta^N \in \Sigma_N : N\ge 1\}$ of configurations and $H\in\mc C^{\infty}_c([0,T]\times \L)$.
There exists a positive constant $C_1$  depending only on $b(\cdot)$ and $\b$ so that   for any given  $\d_0 >0$,  
for  any $\delta$,  $0\le \delta\le \delta_0$,
\begin{equation}\label{aaa}
\limsup_{\epsilon \to 0} \limsup_{N\to\infty} \frac 1{N^d}  \log
\bb E^{\b}_{\eta^N} \Big [ 
\exp\Big(\delta \, N^d \sup_{1\le i\le d}\tq_{i,H}^{\d_0}(\pi^{N,\epsilon}\big)\Big) \Big] 
\; \le \; C_1( T+1)\;.
\end{equation}
\end{corollary}

\begin{proof}
From the following inequality
\begin{equation}
\label{maxlim}
\limsup_{N\to\infty} \frac{1}{N^d}\log(a_N+b_N) 
\leq \max\left\{\limsup_{N\to\infty} 
\frac{1}{N^d}\log a_N,\limsup_{N\to\infty} \frac{1}{N^d}\log b_N\right\}\, ,
\end{equation}
the limit in \eqref{aaa} is bounded above by
$$
\max_{1\le i\le d}\Big\{
\limsup_{\epsilon \to 0} \limsup_{N\to\infty} \frac 1{N^d}  \log
\bb E^{\b}_{\eta^N} \Big [ 
\exp\Big(\delta \, N^d \tq_{i,G}^{\d_0}(\pi^{N,\epsilon}\big)\Big) \Big] .   
$$
By  Lemma \ref{hs02} the thesis follows. 
\end{proof}

\medskip
\subsection{The functional $\mc E^\gamma$.} 
For $(G,\varphi)\in \mc C_0^2(\L)\times  \mc C^1([0,T])$ and
$H\in \mc C^\infty_c([0,T]\times \L)$, denote by ${\mf E}_{(G,\varphi)}^{\g,H}$ the functional
\begin{equation}\label{mfe}
{\mf E}_{(G,\varphi)}^{\g,H}(\mb W,\pi) =\sup_{1\le i\le d}\big\{ {\mc Q}_{i,H}(\pi)\big\} 
+\sup_{0\le t\le T}\big\{\V_{(G,\varphi)}^{t,\g}(\mb W,\pi)\big\}\, ,
\end{equation}
where ${\mc Q}_{i,H}(\pi)={\tilde{\mc Q}}_{i,H}^{2}(\pi)$ with $\d_0=2$, and 
$\V_{(G,\varphi)}^{t,\g}$ are defined in  \eqref{ccc} and \eqref{eq:4.5b}. 

Next lemma shows that we may set
the large deviations rate functional equal to $+\infty$ on the set of paths $(W,\pi)$ which do not belong to 
$\big\{ (\mb W,\pi)\ :\ {\mc E}^\gamma(\mb W,\pi) <+\infty \big\}$.

\begin{lemma}
\label{mf01}  Fix a sequence $\{\eta^N \in \Sigma_N : N\ge 1\}$ of configurations, 
$(G,\varphi)\in \mc C_0^2(\L)\times  \mc C^1([0,T])$ and $H\in \mc C^\infty_c([0,T]\times \L)$.
There exists a positive constant $C_2$  depending only on $b(\cdot)$ and $\b$ so that,
for  any $0\le \delta\le 1$ and any $1\le i\le d$, 
\begin{equation*}
\limsup_{\epsilon \to 0} \limsup_{N\to\infty} \frac 1{N^d}  \log
\bb E^{\b}_{\eta^N} \Big [ 
\exp\Big(\delta \, N^d {\mf E}_{(G,\varphi)}^{\g,H}(\mb W^N,\pi^{N,\epsilon})\Big) \Big] 
\; \le \; C_2( T+1)\;.
\end{equation*}
\end{lemma}

\begin{proof} By Schwarz inequality,
\begin{equation*}
\begin{aligned}
&\frac 1{N^d}  \log
\bb E^{\b}_{\eta^N} \Big [ 
\exp\Big(\delta \, N^d {\mf E}_{(G,\varphi)}^{\g,H}(\mb W^N,\pi^{N,\epsilon})\Big) \Big] \\
&\qquad \  \; \le \;\frac 1{2N^d}  \log
\bb E^{\b}_{\eta^N} \Big [ 
\exp\Big(2\delta \, N^d \sup_{0\le t\le T}\big\{\V_{(G,\varphi)}^{t,\g}(\mb W^N,\pi^{N,\epsilon})\Big) \Big] \\
&\qquad \ \   \ +  \;\frac 1{2N^d}  \log
\bb E^{\b}_{\eta^N} \Big [ 
\exp\Big(2\delta \, N^d \sup_{1\le i\le d}\big\{ {\mc Q}_{i,H}(\pi^{N,\varepsilon})\big\}\Big) \Big] \,.
\end{aligned}
\end{equation*}
The result is an imediate consequence of Lemma \ref{t-u2} and of Corollary \ref{hs03}.
\end{proof}

\bigskip
\subsection{Upper bound}
In this section we investigate the upper bound of the large deviations principle for compact sets and then for closed sets
of the couple $(\mb W^N,\pi^N)$ on the topological space $D([0,T),\mc M^{d+1} )$. 
We follow the strategy of \cite{mo3}, relying
on some properties of the rate function that we proved in the last subsections.
Notice however that in the present case
the proof is slightly more demanding due to the definition of the energy functional $\mc E^\g$.
We first prove an upper bound with an auxiliary rate functional.

Recall from \eqref{mfe} the definition of ${\mf E}_{(G,\varphi)}^{\g,H}$. We introduce the
functional ${\mf E}^\g:D([0,T],\mc M^{d}\times {\mc F}^1)\to [0,+\infty]$ defined by
\begin{equation}\label{mce0}
{\mf E}^\g(\mb W,\pi)\, =\,
\sup_{G,\varphi,H}\Big\{{\mf E}_{(G,\varphi)}^{\g,H}(\mb W,\pi)\Big\}\, ,
\end{equation}
where the supremum is carried over all 
$(G,\varphi, H)\in \mc C_0^2(\L)\times  \mc C^1([0,T])\times \mc C^\infty_c([0,T]\times \L) $.
Notice that ${\mf E}^\g(\mb W,\pi)<+\infty$ if and only if $\mc E^\g (\mb W,\pi)< +\infty$.

For each $0\le a\le 1$,
let ${\mf F}_a:D([0,T],\mc M^{d+1})\to [0,+\infty]$ be the functional given by
\begin{equation*}
{\mf F}_a(\mb W,\pi)  =
\begin{cases}
  {\J}_T(\mb W,\pi) +a \, {\mf E}^\g(\mb W,\pi) \ & \hbox{ if } \ \ 
      D([0,T],\mc M^{d}\times {\mc F}^1)\, ,\\ 
+\infty & \hbox{ otherwise .}
\end{cases}
\end{equation*}

\begin{proposition}
\label{up1}
Let  ${\mc K}$ be a compact set of $D([0,T],\mc M^{d+1})$. 
There exists a positive constants $C_2$, such that for any $0< a\le 1$,
\begin{equation*}
\limsup_{N\to\infty} \frac 1{N^d}  \log
 Q^{\b}_{\eta^N} ({\mc K})\;\le\; -\frac1{1+a}\inf_{(W,\pi)\in {\mc K}}
 {\mf F}_a(\mb W,\pi) \, +\, \frac{a}{1+a} C_2 (T+1)\; .
\end{equation*}
\end{proposition}

\begin{proof} Fix a compact set ${\mc K}$ of $D([0,T],\mc M^{d+1})$ and functions 
$(G,\varphi)\in \mc C_0^2(\L)\times  \mc C^1([0,T])$,
$H\in \mc C^\infty_c([0,T]\times \L)$, 
$\mb V = (V_1,\cdots, V_d) \in (\mc C^{1,1}([0,T]\times \L))^d$. Denote by $\mb O$ the vector-valued function $(0,\cdots,0)$,
where each component is the zero function and 
recall from \eqref{vne} and \eqref{vneb}, the definition
of ${\mc G}_{N,\epsilon}^{\mb V,\mb 0,\b}$ and ${\mc H}_N^{\partial_1 V_1}$.
For $\delta>0$, let $B_{N,\epsilon,\d}^{\mb V,\mb 0,\b}$, $E_{N,\d}^{\partial_1 V_1}$ be the sets 
of trajectories $(\eta_t)_{t\in [0,T]}$ defined by
\begin{eqnarray*}
B_{N,\epsilon,\d}^{\mb V,\mb 0,\b} &=& \Big\{ \eta_\cdot \in D([0,T], {\Sigma}_N) : 
\Big| \int_0^T {\mc G}_{N,\epsilon}^{\mb V,\mb 0,\b}(t,\eta_t) dt \Big| \le \d\Big\} \;,\\
E_{N,\d}^{\partial_1 V_1}&=& \Big\{ \eta_\cdot \in D([0,T], {\Sigma}_N) : 
\Big| \int_0^T {\mc H}_N^{\partial_1 V_1}(t,\eta_t) dt \Big| \le \d\Big\} \;
\end{eqnarray*}
and set 
\begin{equation*}
A_{N,\epsilon,\d}^{\mb V,\b} \;=\;
B_{N,\epsilon,\d}^{\mb V,\mb 0,\b}\cap
E_{N,\d}^{\partial_1 V_1}\; .
\end{equation*}

By \eqref{maxlim} and the superexponential estimates 
stated in Proposition \ref{see}, for any $\d>0$ 
\begin{equation}\label{ub0}
\limsup_{\epsilon\to 0}\limsup_{N\to\infty} \frac 1{N^d}  \log
 Q^{\b}_{\eta^N} \Big({\mc K} \cap \big(A_{N,\epsilon,\d}^{\mb V,\b}\big)^c \Big)\;=\; 
-\infty\; ,
\end{equation}
where $\big(A_{N,\epsilon,\d}^{\mb V,\b}\big)^c$ stands for the complementary of the set
$A_{N,\epsilon,\d}^{\mb V,\b}$.

Recall from \eqref{mfe} the definition of 
${\mf E}_{(G,\varphi)}^{\g,H}$. To short notation we denote by
${\widehat {\mc K}}_{N,\epsilon,\d}^{\mb V,\b} ={\mc K}\cap A_{N,\epsilon,\d}^{\mb V,\b} $, and
write
\begin{equation*}
\begin{aligned}
&\frac 1{N^d}  \log
 Q^{\b}_{\eta^N} \Big({\mc K} \cap A_{N,\epsilon,\d}^{\mb V,\b} \Big)
\; =\;\\
&\qquad
\frac 1{N^d}  \log
  \Es_{\eta^N}^{\b}\Big[ \1\{{\widehat {\mc K}}_{N,\epsilon,\d}^{\mb V,\b}  \}
  e^{-\frac{a}{1+a}N^d {\mf E}_{(G,\varphi)}^{\g,H}(\mb W^N,\pi^{N,\epsilon}) } 
e^{\frac{a}{1+a}N^d {\mf E}_{(G,\varphi)}^{\g,H}(\mb W^N,\pi^{N,\epsilon}) }\Big]\, .
\end{aligned}
\end{equation*}
By  H\"older inequality the right hand side of the last equality is bounded above by
\begin{equation}\label{ub1}
 \begin{aligned}
& \frac{1}{1+a}\frac{1}{N^d} \log
  \Es_{\eta^N}^{\b}   \Big[ \1\{{\widehat {\mc K}}_{N,\epsilon,\d}^{\mb V,\b} \}
  e^{-aN^d {\mf E}_{(G,\varphi)}^{\g,H}(\mb W^N,\pi^{N,\epsilon})}\Big]\\
&\qquad\quad
+ \; \frac{a}{1+a}\frac{1}{N^d} \log
 \Es_{\eta^N}^{\b} \Big[ e^{N^d{\mf E}_{(G,\varphi)}^{\g,H}(\mb W^N,\pi^{N,\epsilon})}\Big]\; .
 \end{aligned}
\end{equation}
From Lemma \ref{mf01}, the limsup when $N\uparrow \infty$ and $\epsilon \downarrow 0$ of 
the second term of this inequality is bounded by $ \displaystyle\frac{a}{1+a} C_2  (T+1)$, 
while the first term can be rewriten as the expectation with respect to 
the perturbed process introduced in Subsection 
\ref{s-perturb} whose law is given by $\mathbb P_{\eta^N}^{\b,\mb V}$, that is
\begin{equation}\label{ub2}
\frac{1}{1+a}\frac{1}{N^d}  
  \log \Es_{\eta^N}^{\b,\mb V}   \Big[\frac{d{\mathbb P}_{\eta_N}^{\b}}{d{\mathbb P}_{\eta_N}^{\b,\mb V}}
  \1\{{\widehat {\mc K}}_{N,\epsilon,c}^{F,\mb V,\b}  \}\\
e^{-aN^d {\mf E}_{(G,\varphi)}^{\g,H}(\mb W^N,\pi^{N,\epsilon})}\Big]\, .
\end{equation}
By  \eqref{girsanov},  the Radon-Nikodym derivative of ${\mathbb P}_{\eta_N}^{\b}$ with respect to the probability ${\mathbb P}_{\eta_N}^{\b,\mb V}$ 
defined by the Girsanov formula satisfies on the set $A_{N,\epsilon,\d}^{\mb V,\b}$
\begin{equation*}\label{Mtgb}
\frac{d{\mathbb P}_{\eta_N}^{\b}}{d{\mathbb P}_{\eta_N}^{\b,\mb V}}
\;=\;  \exp N^d\Big\{ - {\widehat \J}_{\mb V}^T (\mb W^N,\pi^{N,\epsilon}) 
\; +\; r(N,\epsilon,\d,\mb V) \Big\} \, ,
\end{equation*}
where  $ {\widehat \J}_{\mb V}^T(\cdot)$ is the functional 
defined in \eqref {eq:4.6}, and $r(N,\epsilon,c,\mb V)$ is a quanity satisfying 
\begin{equation*}\label{r-ub}
\lim_{\d\to 0}\lim_{\epsilon\to 0}\lim_{N\to \infty} r(N,\epsilon,\d,\mb V)=0\, .
\end{equation*}

We now exclude paths whose densities are not absolutely continuous with respect to the Lebesgue measure.
Fix a sequence $\{f_k: k\geq 1\}$ of smooth nonnegative functions
dense in $\mc C^0 (\Lambda)$ for the uniform topology.  For
$k\ge 1$ and $\varrho>0$, let
\begin{equation*}
\begin{aligned}
&\mb D_{k,\varrho} = \Big\{(\mb W,\pi)\in  D([0,T],\mc M^{d+1}):  \\
&\qquad\qquad\qquad\qquad\qquad
\, 0\le <\pi_t , f_k>
\le \int_{\L} f_k(x) \, dx  \,+\, C_k \varrho
\;,\,\  0\le t\le T \Big\} \, ,
\end{aligned}
\end{equation*}
where $C_k = C(\Vert \nabla f_k \Vert_\infty)$ is a constant depending on the
gradient $\nabla f_k$ of $f_k$. The sets $\mb D_{k,\varrho}$, $k \ge 1$,  $\varrho >0$ are closed subsets of 
$\in D([0,T],\mc M^{d+1})$, as well as 
$$
\mf D_{m,\varrho} \;=\; \bigcap_{k=1}^m \mb D_{k,\varrho}\;, \quad m\ge 1\, .
$$
Note that the empirical measure $\pi^N$  belongs to $\mf D_{m,\varrho}$ for $N$ sufficiently large. We have that
\begin{equation}
\label{mbD} 
{D([0,T],\mc M^{d}\times \mc F^1) = \cap_{n\ge 1} \cap_{m\ge 1}
\mf D_{m,1/n}}. 
\end{equation}   

For $m,n\in\Z_+$, let ${\widehat {\mc E}}_{(G,\varphi),H}^{\g,\epsilon,m,n}:D([0,T],\mc M^{d+1})\to\R\cup\{\infty\}$ be the functional given by
\begin{equation}
\label{F-mn}
{\widehat {\mc E}}_{(G,\varphi),H}^{\g,\epsilon,m,n}(\mb W,\pi)  =
\begin{cases}
  {\mf E}_{(G,\varphi)}^{\g,H}(\mb W,\pi^{\epsilon}) \ & \hbox{ if } \ \ 
      \pi\in \mf D_{m,\frac{1}{n}}\, ,\\ 
+\infty & \hbox{ otherwise .}
\end{cases}
\end{equation}
It is lower semicontinuous because so is $(\mb W,\pi) \mapsto {\mf E}_{(G,\varphi)}^{\g,H}(\mb W,\pi^{\epsilon})$,
and because $\mf D_{m,1/n}$ is closed.

Recollecting all previous estimates. Using the inequality \eqref{maxlim}, optimizing over $\pi$ in ${\mc K}$ and 
letting $N\uparrow \infty$, we obtain that, for any $m,n\in\Z_+$, $0<a \le 1$, $\d>0$ and $\epsilon$ small enough
\begin{equation}\label{ub30}
\limsup_{N\to \infty}\frac 1{N^d}  \log
 Q^{\b}_{\eta^N} \Big({\mc K} \Big)
 \, \le\,\\
\frac{1}{1+a}\sup_{(\mb W,\pi)\in {\mc K} }{\widehat {\mf S}}_{\mb V,H,G,\varphi}^{a,\d,\epsilon,m,n}(\mb W,\pi)\, .
 \end{equation}
Here, we have denoted
\begin{equation*}
{\widehat {\mf S}}_{\mb V,H,G,\varphi}^{a,\d,\epsilon,m,n}(\mb W,\pi)\, =\, 
\max \Big\{ 
 \Big(-
{{\widehat \J}^T}_{\mb V}(\mb W,\pi^\epsilon)+a {\mf R}_{H,G,\varphi}^{a,\d,\epsilon,m,n}(\mb W,\pi)\Big)
\, ,\, U_{0,a}(\mb V,\epsilon) \Big\}\, ,
 \end{equation*}
where
\begin{equation*}
 \begin{aligned}
{\mf R}_{H,G,\varphi}^{a,\d,\epsilon,m,n}(\mb W,\pi) &= 
 -{\widehat {\mc E}}_{(G,\varphi),H}^{\g,\epsilon,m,n}(\mb W,\pi) 
  +U_{1,a}(G,\varphi,H,\epsilon)+r(N,\epsilon,\d,\mb V) \, ,\\
  U_{1,a}(G,\varphi,H,\epsilon) & = \limsup_{N \to \infty}\frac{1}{N^d} \log
 \Es_{\eta^N}^{\b} \Big[ e^{N^d{\mf E}_{(G,\varphi)}^{\g,H}(\mb W^N,\pi^{N,\epsilon})}\Big]\; ,\\ 
U_{0,a}(\mb V,\epsilon) & =(1+a) \limsup_{N\to\infty} \frac {1}{N^d}  \log
 Q^{\b}_{\eta^N} \Big({\mc K} \cap \big(A_{N,\epsilon,\d}^{\mb V,\b}\big)^c \Big)\, .\\
 \end{aligned}
\end{equation*}

Note that, for each $m,n\in\Z_+$, $0<a \le 1$, $\d>0$ and $\epsilon>0$, the functional
${\widehat {\mf S}}_{\mb V,H,G,\varphi}^{a,\d,\epsilon,m,n}$ is lower semicontinuous. Minimizing the right hand side of
the inequality \eqref{ub30} over  $m,n\in\Z_+$, $\d>0$ and $0<\epsilon<1$, and using Lemma
A2.3.3 in \cite{kl} for our compact $\mc K$, we get 
$$
\limsup_{N\to \infty}\frac 1{N^d}  \log
 Q^{\b}_{\eta^N} \Big({\mc K} \Big)
 \, \le\,\\
\frac{1}{1+a}\sup_{(\mb W,\pi)\in {\mc K} }\inf_{\d,\epsilon,m,n}
{\widehat {\mf S}}_{\mb V,H,G,\varphi}^{a,\d,\epsilon,m,n}(\mb W,\pi)\, .
$$
By \eqref{ub0}, \eqref{r-ub}, \eqref{mbD} and Lemma \ref{mf01}
\begin{equation*}
\limsup_{\d\to 0} \limsup_{\varepsilon\to 0}\limsup_{m\to\infty}\limsup_{n\to\infty} 
{\widehat {\mf S}}_{\mb V,H,G,\varphi}^{a,\d,\epsilon,m,n}(\mb W,\pi) 
\le 
-\mf F_{{}_{\mb V,H}}^{(G,\varphi),a} (\mb W,\pi)
+aC_2(T+1)\, ,
\end{equation*}
where
\begin{equation*}
\mf F_{{}_{\mb V,H}}^{(G,\varphi),a}(\mb W,\pi)
=
\begin{cases}
  {{\widehat \J}}_{\mb V}^T (\mb W,\pi)  +a {\mf E}_{(G,\varphi)}^{\g,H}(\mb W,\pi)
  \ & \hbox{ if } \ \ 
      \pi\in D([0,T],{\mc M}^d\times {\mc F}^1)\, ,\\ 
+\infty & \hbox{ otherwise .}
\end{cases}
\end{equation*}
This result and the last inequality imply, 
\begin{equation*}
 \begin{aligned}
&\limsup_{N\to \infty}\frac 1{N^d}  \log
 Q^{\b}_{\eta^N} \Big({\mc K} \Big)\\
&\qquad\qquad  \, \le\,
-\frac{1}{1+a}\inf_{(\mb W,\pi)\in {\mc K} }\Big\{ 
\mf F_{{}_{\mb V,H}}^{(G,\varphi),a}(\mb W,\pi)
\Big\}
+\frac{a}{1+a}C_2(T+1)\, ,
  \end{aligned}
\end{equation*}
for any $\mb V,H,G,\varphi$. 
To conclude the proof of the proposition, it remains to 
 Minimize the last inequality over  $\mb V,H,G,\varphi$, and to use again Lemma
A2.3.3 in \cite{kl} for the compact $\mc K$.
\end{proof}
  
\medskip    
\noindent{\it Proof of the upper bound.}
Denote by ${\widehat {\mc E}}^\g:D([0,T],\mc M^{d+1})$ the lower semicontinuous functional 
\begin{equation*}
{\widehat {\mc E}}^\g(\mb W,\pi)  =
\begin{cases}
  {\mf E}^\g(\mb W,\pi)  \ & \hbox{ if } \ \ (\mb W,\pi)\in
      D([0,T],\mc M^{d}\times {\mc F}^1)\, ,\\ 
+\infty & \hbox{ otherwise .}
\end{cases}
\end{equation*}

Let ${\mc K}$ be a compact set of $D([0,T],\mc M^{d+1})$. If for all $(\mb W,\pi)\in {\mc K}$, 
${\widehat {\mc E}}^\g(\mb W,\pi)=+\infty$ then the upper bound is trivially satisfied. 
Suppose that $\displaystyle \inf_{(\mb W,\pi) \in {\mc K}}\big\{{\widehat{\mc E}}^\g(\mb W,\pi)\big\}<\infty$, 
from Proposition \ref{up1}, for any $0< a\le 1$,
\begin{equation*}
\begin{aligned}
&\limsup_{N\to\infty} \frac 1{N^d}  \log
 Q_{\eta^N}^{\b} \big({\mc K}\big)\;\le\;
 -\frac1{1+a}\inf_{(\mb W,\pi)\in {\mc K}\atop  {{\widehat{\mc E}}_\g}(\mb W,\pi)<\infty}\mf F_a(\mb W,\pi)\, +\,  \frac{a}{1+a} C_2  (T+1) \\
&\qquad
=\; -\frac1{1+a}\inf_{(\mb W,\pi)\in {\mc K}}\Big\{ {\mc J}_T^\g (\mb W,\pi)\, +\,a{\widehat{\mc E}}_\g(\mb W,\pi) \Big\}
\, +\, \frac{a}{1+a} C_2  (T+1)\\
&\qquad
\le\; -\frac1{1+a}\inf_{(\mb W,\pi)\in {\mc K}} {{\mc J}_T}^\g (\mb W,\pi) -
\frac{a}{1+a}\inf_{(\mb W,\pi)\in {\mc K}}{\widehat{\mc E}}_\g(\mb W,\pi) +   \frac{a}{1+a} C_2  (T+1) \,  .
\end{aligned}
\end{equation*}
To conclude the proof of the upper bound for compact sets, it remains to let $a\downarrow 0$.

To pass from compact sets to closed sets, we have to obtain
exponential tightness for the sequence  $\big\{Q_{\eta^N}^{\b}\, ,\; N\ge 1\big\}$.
The proof presented in \cite{bdgjl2, bl2} is easily adapted to our context.
\cqfd

\medskip
\section{large deviations lower bound for the empirical current}\label{secldlb}
The strategy of the proof of the lower  bound consists of two steps.
We first get a lower bound for neighbourhoods of regular trajectories. Then we extend the lower bound
for all open set by showing in Theorem \ref{th-I-d} that the set of all regualar trajectories is $\mc J_T^\g$-dense in the following sens:

\begin{definition}
A subset $\ca$ of $D([0,T],\mc M^{d+1})$ is said to be
${\mc J}_T^\g$-dense if for every $(\mb W,\pi)$ in $D([0,T],\mc M^{d+1})$ such
that ${\mc J}_T^\g(\mb W,\pi)<\infty$, there exists a sequence $\{(\mb W^n,\pi^n) : n\ge
1\}$ in $\ca$ such that $(\mb W^n,\pi^n)$ converge to $(\mb W,\pi)$ in $D([0,T], \mc M^{d+1})$ and
$\displaystyle \lim_{n\to\infty}{\mc J}_T^\g(\mb W^n,\pi^n)  \, =\, {\mc J}_T^\g(\mb W,\pi)$.
\end{definition}

To clarify the meaning of regular trajectory, we 
consider the heat equation given by the boundary  
value problem \eqref{eq:1} for $\b=0$: 
\begin{equation}
\label{f00}
\begin{cases}
\partial_t\rho  & = \;\;  \Delta\rho \, \qquad \hbox {in } \qquad  \Lambda \times (0,T) ,\\ 
\rho_0 (\cdot)    & =  \;\; \gamma (\cdot)  \,  \qquad \hbox {in } \qquad  \Lambda ,\\
\rho_t|_\Gamma  & = \;\; b(\cdot) \;\;\; 
\hbox{ for }\; 0\leq t\leq T \, .
\end{cases}
\end{equation}
Denote by  $\rho^{(0)}$ its unique weak solution, and set
$\pi_t^{(0)}(du)=\rho^{(0)}_t(u)du$. Let $(\mb W^{(0)}_t)_{t\in [0,T]}$ be the weak solution of the equation
$$
\partial_t \mb W_t+\nabla \rho^0 =0\, , \qquad \mb W_0=0\, .
$$
Notice that, an approximation of 
$\frac{\nabla \rho^{(0)}}{\sigma(\rho^{(0)})}$ by smooth functions shows that 
$\mc Q(\rho^{(0)})<\infty$, (see \cite{blm}, (5.1)). Morover,
by construction $(\mb W^0,\pi^{(0)})\in \mf A_\g$, and 
\begin{equation*}
\mc J^\g(\mb W^0,\pi^0)\, \le\,
\frac{\b^2}{4}\int_0^T dt\int_\L\big|\nabla\rho^{(0)}_t\big|^2\, <\, \infty\; .
\end{equation*}
(see \cite{mo3}, Lemma 5.8.).

\begin{definition} A trajectory $(\mb W,\pi) \in D([0,T],\mc M^{d+1})$ is said to be regular
if 
\begin{itemize}  
 \item[(i)] ${\mc J}_T^\g(\mb W,\pi)<\infty$, 
 $\pi(t,du)=\rho_t(u)du$.
 \item[(ii)]There exists $c>0$ such that $(\mb W,\pi)=(\mb W^0,\pi^0)$ in the time interval $[0,c]$.
\item[(iii)] For all $0<\delta \le T$, there exists $\epsilon >0$ such
that $\epsilon \le \rho_t(u) \le 1-\epsilon$ for $(t,u)\in [\delta, T]\times \Lambda$.
\item[(iv)] $\rho$ is the solution of the  boundary value problem  \eqref{eq:V} for some 
$\mb V=(V_1,\cdots,V_d)\in  \big(\mc C^{1,1}([0,T]\times \L)\big)^d$.
\end{itemize}
We denote by $\mc A^0$ the class of all regular trajectories. 
\end{definition}

To derive the lower bound for paths $(\mb W, \pi)$ in $\mc A^0$ we follow the  arguments
used in \cite{kov, kl} to show that for each neighborhood $\mc N_{(\mb W,\pi)}$ of
$(\mb W,\pi)$ in  $D \big( [0,T], \mc M^{d+1} \big)$, 
\begin{equation}\label{low1}
\liminf_{N \to \infty}\;\frac1{N^d} \log \mathbb P_{\eta^N}^{\b} \big\{ \mc N_{(\mb W,\pi)} \big\} \ge - 
\mc J_T^\b(\mb W,\pi)\, .
\end{equation}

As mentioned at the beginning of this section, the lower bound of the large deviations principle
is then accomplished for general trajectotries using the next result.

\begin{theorem}\label{th-I-d} The class $\mc A^0$ is ${\mc J}_T^\g$-dense.
\end{theorem}

The proof of this theorem is an adaptation of the $I$-density presented in \cite{blm, flm, mo3} for the couple $(\mb W,\pi)$.
We therefore provide only a presentation of its main steps, with an outline of the proofs.

\begin{lemma} The set of all trajectories satisfying (i) and (ii) is ${\mc J}_T^\g$-dense.
\end{lemma}

\begin{proof}
Fix a path $(\mb W,\pi)$ such that ${\mc J}_T^\g(\mb W,\pi)<\infty$.  
For $\epsilon
>0$, define $(\mb W^\epsilon,\pi^\epsilon)$ as
\begin{equation*}
\big(\mb W^\epsilon_t,\pi^\epsilon_t\big)
\, =\,
\left\{
\begin{array}{ll}
{\displaystyle
  \big(\mb W^{(0)}_t, \pi^{(0)}_t\big)} & \text{for $0\le t\le \epsilon$}, \\
{\displaystyle
\big(\mb W^{(0)}_{2\epsilon - t}, \pi^{(0)}_{2\epsilon - t}\big)} & \text{for $\epsilon \le t\le 2
  \epsilon$}, \\
{\displaystyle
\big(\mb W_{t-2\epsilon}, \pi_{t-2\epsilon}\big)} & \text{for $2 \epsilon \le t\le T$}. 
\end{array}
\right.
\end{equation*}
Clearly, $ \lim_{\ve \to 0}\big(\mb W^\epsilon,\pi^\epsilon\big) =\big(\mb W,\pi\big)$ in 
$D([0,T],\cm^{d+1})$.
The same strategy as in Lemma 5.4., \cite{blm} or Lemma 5.11., \cite{mo3},  yields
$\mc J_T^\g \big(\mb W^\epsilon,\pi^\epsilon\big)<\infty$, for all $\varepsilon >0$,  and
$\displaystyle\lim_{\varepsilon\to 0} \mc J_T^\g \big(\mb W^\epsilon,\pi^\epsilon\big)\, =\, 
\mc J_T^\g \big(\mb W,\pi\big)$.
This concludes the proof.
\end{proof}

\begin{lemma} The set of all trajectories satisfying (i), (ii) and (iii) is ${\mc J}_T^\g$-dense.
\end{lemma}

\begin{proof}
Denote by $\mc A^1$ the set of all trajectories $(\mb W,\pi)$ satisfying (i), (ii) and (iii). 
By the previous lemma, it is enough to show that each trajectory $(\mb W,\pi)$ satisfying (i) and (ii)
can be approximated by trajectories in $\mc A^1$. Fix such trajectory $ (\mb W, \pi) $.
For each $0<\varepsilon\le1$, let $(\mb W^\varepsilon,\pi^\varepsilon)$ given by
$$
\mb W^\varepsilon=(1-\varepsilon)\mb W+\varepsilon \mb W^{(0)}\, ,\quad \pi^\varepsilon=(1-\varepsilon)\pi+\varepsilon \pi^{(0)}\, . 
$$  
Repeating the arguments presented in \cite[Lemma 5.12.]{mo3}, one can prove that 
$ \lim_{\ve \to 0}\big(\mb W^\epsilon,\pi^\epsilon\big) =\big(\mb W,\pi\big)$ in $D([0,T],\cm^{d+1})$,
$\mc J_T^\g \big(\mb W^\epsilon,\pi^\epsilon\big)<\infty$, for all $\varepsilon >0$  and
$\displaystyle\lim_{\varepsilon\to 0} \mc J_T^\g \big(\mb W^\epsilon,\pi^\epsilon\big)\, =\, 
\mc J_T^\g \big(\mb W,\pi\big)$.
\end{proof}

\medskip
\noindent {\it Proof of Theorem \ref{th-I-d}. }
Recall that $\mc A^1$ stands for the set of all trajectories $(\mb W,\pi)$ satisfying (i), (ii) and (iii). 
From the previous lemmata, it is enough to show that each trajectory $(\mb W,\pi)$ in $\mc A^1$
can be approximated by trajectories of $\mc A^1$ satisfying (iv). Fix $(\mb W,\pi)\in \mc A^1$ and denote $\rho_t(\cdot)$ the density of $\pi_t$ for $0\le t\le T$.
By Lemma \ref{rep0}, there exist $\mb U=(U_1,\cdots,U_d)\in \LL^2(\sigma(\pi))$ and $F\in H^1_0(\sigma(\pi))$
such that $\rho$ solves the equation \eqref{eq:V} with $\mb V=\nabla F$ and $\mb W$ solves the equation
\eqref{rep2}. We claim that $\mb U\in\big( L^2([0,T]\times \L)\big)^d$ and $F\in  L^2([0,T],H^1(\L))$. Indeed,
from condition (ii), $\rho$ is the
weak solution of \eqref{f00} in some time interval $[0,2\delta]$ for
some $\delta>0$. In particular, $\rho_t = \rho_t^{(0)}$, 
$\displaystyle \mb W_t=-\int_0^t \nabla\rho_s^{(0)} ds$ for $0\le t\le 2\d$, which implies that
$\mb U_t = \nabla F_t=-\b \nabla (\jn\star \rho_t)$ a.e in
$[0,2\delta]\times\Omega$. On the other hand, from condition (iii),
there exists $\epsilon>0$ such that $\epsilon \le
\rho_t(\cdot) \le 1-\epsilon$ for $\delta \le t\le T$.
Hence, by Lemma \ref{compar1},
\begin{equation*}\label{g05b}
\begin{aligned}
\int_0^T dt\int_\L \big|\mb U(t,u)\big|^2 du &\le
\int_0^\d dt\int_\L {\b^2} |\nabla (\jn\star \rho_t)(u)\big|^2 du +
\frac{1}{\sigma(\ve)}\Vert U\Vert_{\LL^2(\sigma(\pi))}^2\\
\ &\ \le {\b^2}\int_0^T dt\int_\L  \big|\nabla \rho_t(u)\big|^2 du +
\frac2{\sigma(\ve)} \mc J_T^\g(\mb W,\pi) <\infty \, ,\\
\int_0^T dt\int_\L \big|\nabla F_t(u)\big|^2 du & \, \le\, 
\int_0^\d dt\int_\L {\b^2} \big|\nabla (\jn\star \rho_t)(u)\big|^2 du
+ \frac{1}{\sigma(\ve)}\Vert F\Vert_{H_0^1(\sigma(\pi))}^2\\
\ &\ \le {\b^2}\int_0^T dt\int_\L  \big|\nabla \rho_t(u)\big|^2 du +
\frac2{\sigma(\ve)} {\mc I}_T^\g(\pi) <\infty \, .
\end{aligned}
\end{equation*}

Let $\big\{U^n=(U^n_1,\cdots,U^n_d) ,\,\; n\!\ge \! 1\big\}\subset (C^{1,1}([0,T]\times \L))^d$ and 
$\big\{F^n ,\,\; n\ge 1\big\}\subset C^{1,2}([0,T]\times \L)$
be two sequences of functions such that $\displaystyle\lim_{n\to +\infty} U^n=U$ in
$\big(L^2\big( [0,T]\times \L\big)\big)^d$, and $\displaystyle\lim_{n\to +\infty} F^n=F$ in $L^2\big( [0,T],H^1(\L)\big)$.
For each integer $n>0$, let $\mb W^n$ be the weak solution of the equation
$$
\partial_t \mb W_t\, =\, -\nabla \rho_t^n +\sigma(\rho_t^n) \big[\b \nabla (\jn  \star \rho_t^n )  \, +\, \mb U^n\big]
\, , \qquad \mb W_0=0\, ,
$$
where $\rho^n$ is the weak solution of
\eqref{eq:V} with $\nabla F^n$ in place of $\mb V$. 
We set $\pi^n (t,du) =
\rho^n(t,u)du$. 

We examine in this paragraph the energy $\mc E^\g(\mb W^n, \pi^n)$. Since
$\rho^n$ solves the equation \eqref{eq:V}, by Lemme \ref{rep0}, for any $G\in \mc C_0^{1}(\L)$ and any $t\in [0,T]$,
\begin{equation*}
 \begin{aligned}
\< \mb W^n_t ,\nabla G \> &\; 
= \< \pi_t^n ,G \> - \< \pi_0^n ,G \>
+  \int_0^t ds \<(-\nabla F_s^n+U_s^n),\nabla G \>\\
\ &\; = \< \pi_t^n ,G \> - \< \pi_0^n ,G \>\, .
 \end{aligned}
\end{equation*}
On the other hand, since $\sigma(\pi^n)$ is bounded above by $1/2$, from Lemma \ref{compar1}
\begin{equation*}
 \begin{aligned}
{\mc I}_T^\g(\pi^n) &
= \frac{1}{2}\int_0^T\ dt\;
\< \sigma(\rho^n_t),\nabla F^n_t \cdot  \nabla F^n_t \>\\
\ & \leq \mc J^\g(\mb W^n,\pi^n) = \frac{1}{2}\int_0^T\ dt\;
\< \sigma(\rho^n_t) , U^n_t \cdot  \nabla U^n_t \>\\
&\ \le C_0 \big\| U^n \big\|_{L^2([0,T]\times \L)}  .
\end{aligned}
\end{equation*}
In particular, $\{{\mc I}_T^\g(\pi^n),\; n\ge 1\}$ and  $\{\mc J_T^\g(\mb W^n,\pi^n),\; n\ge 1\}$ are
uniformly bounded.
Thus, Lemma \ref{lem01},  implies the uniform boundedness of the
sequence $\{\mc Q(\pi^n),\; n\ge 1\}$. 

In order to extract a converging subsequence from the sequence $\big\{(\mb W^n,\pi^n),\; n\ge 1 \big\}$, we need to show
the relative compactness of the set $\big\{(\mb W^n,\pi^n),\; n\ge 1 \big\}$ in the topological space
$D([0,T],\mc M^{d+1})$. For each $n\ge 1$,  denote
by $|\mb W^n|$ (resp. $\|\mb W^n \|$) the variation (resp. total variation)
 of the signed measure $\mb W^n$, and for
 shortness of notation,  denote $\L_T=[0,T]\times \L$. By construction, Schwartz inequality and since $\sigma(\cdot)$
 is bounded by 1/2, 
 \begin{equation*}
 \begin{aligned}
& \sup_{n\ge 1}\sup_{0\le t\le T} \big(\|\mb W^n_t\|+ \|\pi^n_t\|\big)
\le \sqrt{2 T}\sup_{n\ge 1}\Big\{ \|\nabla \rho^n\|_{L^2(\L_T)}\\
\ &\qquad\qquad\qquad\qquad\qquad
+\frac\b2\|\nabla (\jn  \star \rho^n\|_{L^2(\L_T)}
\, +\, \frac{1}2\|\mb U^n\|_{L^2(\L_T)}\Big\}\, <\, \infty\,  .
  \end{aligned}
\end{equation*}
Moreoever, for any $s,t\in [0,T]$,
any $\mb V \in \big(\mc C^{1}(\L) \big)^d$ and any $G \in \mc C^{2}_0(\L)$,
\begin{equation*}
 \begin{aligned}
& \big|\< \mb W^n_t ,\mb V \> +  \<\pi^n_t , G \> -\< \mb W^n_s ,\mb V \> - \<\pi^n_s , G \>\big|\\
&\ \ 
\le \sqrt{|t-s|}  \big\|\mb V\big\|_{L^2(\L)}
\Big\{ \|\nabla\rho^n\|_{L^2(\L_T)}
+\frac{\b}{2}\|\nabla (\jn  \star \rho^n\|_{L^2(\L_T)}
 +\|\mb U^n\|_{L^2(\L_T)}\Big\}\\
&\ \ 
+ \sqrt{|t-s|}  \big\|G\big\|_{L^2(\L)}
\Big\{ \|\nabla\rho^n\|_{L^2(\L_T)}
+{\b}\|\nabla (\jn  \star \rho^n\|_{L^2(\L_T)}
 +\|\nabla F^n\|_{L^2(\L_T)}\Big\} \\
&\quad
\le \sqrt{|t-s|} \, M\, \Big\{ \big\|\mb V\big\|_{L^2(\L)} + \big\|G\big\|_{L^2(\L)}\Big\}
  \end{aligned}
\end{equation*}
where the constant $M= C(\rho, \mb U,F,\b)$ is such that
$$
\sup_{n\ge 1}\Big\{ \big(1+{\b}\big)\|\nabla \rho^n\|_{L^2(\L_T)} +\|\mb U^n\|_{L^2(\L_T)} + \|\nabla F^n\|_{L^2(\L_T)} \Big\}\le M\, .
$$
The relative compactness for the set $\big\{(\mb W^n,\pi^n),\; n\ge 1 \big\}$,
follows from compactness criterium for the Skorohod topology (see  \cite{ek} Theorem 6.3 page 123).

Let $\{(\mb W^{n_k}, \pi^{n_k}):\, k\geq 1\}$ be a subsequence of $\{(\mb W^{n}, \pi^{n}):\, n\geq 1\}$ converging
to some $(\mb W^* ,\pi^*)$ in $D([0,T],\mc M^{d+1})$. We claim that
$(\mb W^*,\pi^*) = (\mb W,\pi)$
and $\displaystyle\lim_{k\to \infty} \mc J_T^\g \big(\mb W^{n_k},\pi^{n_k}\big)\, =\, 
\mc J_T^\g \big(\mb W,\pi\big)$.
On the one hand, $\{\pi^{n_k}:\, k\geq 1\}$ converges weakly to $\pi^*$ in $L^2\big( \L_T\big)$.   
Since $\mc J_T^\g(\mb W^n, \pi^n)$ is uniformly bounded, by Lemma \ref{g06p} and Lemma \ref{compar1}, $\pi^{n_k}$ converges to
$\pi^*$ strongly in $L^2(\L_T)$.
For every $G$ in $\mc
C^{1,2}_0(\L_T)$, we have
\begin{equation*}
\begin{split}
& \langle\pi^{n_k}_T,G_T\rangle - \langle\gamma,G_0\rangle = 
\int_0^T dt\;\langle\pi^{n_k}_t,\partial_tG_t\rangle\\
&\qquad\qquad\qquad
+\int_0^T dt\;\langle \pi^{n_k}_t,
\Delta G_t\rangle
\; \;-\; \int_0^T dt
\int_{\Gamma}b(r) \,  \text{\bf n}_1(r)\, (\partial_{1}F_t)(r) \, dS(r) \\
&\qquad\qquad\qquad
\,  +\int_0^T \langle \nabla G_t,\sigma(\rho_t^{n_k})
\big[ {\b}\nabla (\jn\star \rho_t^{n_k}) + \nabla F_t^{n_k}\big]\rangle \, dt.
\end{split}
\end{equation*}
Letting  $k\to \infty$, we obtain
that $\pi^*$ is a  weak solution of equation \eqref{eq:V} with $\mb V=\nabla F$. Thus, by uniqueness of weak solutions of \eqref{eq:V},
$\pi^*=\pi$.
On the other hand, $\{\mb W^{n_k}:\, k\geq 1\}$ converges weakly to $\mb W^*$ in $L^2\big( \L_T\big)$. Since $\pi^{n_k}$
converges to $\pi^*$ strongly in $L^2(\L_T)$, 
for every $\mb V$ in $(\mc
C^{1,1}(\L_T))^d$ and any $k\ge 1$, we have
\begin{equation*}
\begin{aligned}
&\< \mb W_T^{n_k}, \mb V_T\> \, 
=\, \int_0^T dt \, \< \mb W_t^{n_k}, \partial_t \mb V_t\> \\
&\qquad\qquad
+ \int_0^T dt \, \< \pi_t^{n_k}, \nabla \cdot \mb V_t\>
\;-\; \int_0^T dt
\int_{\Gamma}b(r) \,  \text{\bf n}_1(r)\, V_1(t, r) \, dS(r)  \\
&\qquad\qquad
+  \b\int_0^T \langle \sigma( \rho_t^{n_k}), \mb V_t\cdot \big[\nabla (\jn\star  \rho_t^{n_k})+\mb U^{n_k}\big]\rangle  dt\, .
\end{aligned}
\end{equation*}
Letting  $k\to \infty$, we obtain that $\mb W^*$ is a  weak solution of the equation \eqref{rep2} associated to $\rho^*$ and $\mb U$, where
$\rho^*$ is the density of $\pi^*$. This proves the first part of the claim.
To conclude the proof it remains to prove that $\displaystyle \lim_{k\to \infty}\mc J_T^\g(\mb W^{n_k},\pi^{n_k})=\mc J_T^\g(\mb W,\pi)$.
The sequence $(\rho^{n_k})_{k>0}$ converges to $\rho$ strongly in $L^2(\L_T)$ and the sequence $(\mb U^{n_k})_{k>0}$ converges to
$\mb U$ in $L^2(\L_T)$.  Taking into account that $\rho$ is bounded and $\sigma$ is Lipschitz, we obtain
\begin{equation*}
\begin{aligned}
\lim_{k\to \infty} \mc J_T^\g(\mb W^{n_k},\pi^{n_k})&
=\lim_{k\to \infty} \frac{1}{2}\int_0^T\ dt\;
\langle \sigma(\rho^{n_k}_t), \mb U^{n_k}_t \cdot  \mb U^{n_k}_t\rangle\\
\ & \  = \frac{1}{2}\int_0^T\ dt\;
\langle \sigma(\rho_t),\mb U_t \cdot  \mb U_t\rangle \, =\, \mc J_T^\g(\mb W,\pi)\, .
\end{aligned}
\end{equation*}
This concludes the proof.
\cqfd

\medskip
\section{large deviations for the empirical density}\label{ldp-ed}
In this section we prove Theorem \ref{s02}.
As we mentioned in the introduction, the large deviations principle for the empirical density 
can be recovered from the one for the current. Indeed, it follows from Theorem \ref{gd-(d,c)} and 
the contraction principle, that the rate function ${\widetilde {\mc I}}_T^\g$ for the 
empirical density is given by the variational formula
\begin{equation}
\label{empi-dev}
{\widetilde {\mc I}}_T^\g (\pi) \;=\; \inf_{\mb W\, : (\mb W,\pi) \in \mf A_\g} 
\mc J_T^\g (\mb W, \pi)\;,
\end{equation}
where $\mf A_\g$ is defined by \eqref{eq:4.5}. To conclude the proof of Theorem \ref{s02}, we then need to show that the functional
${\mc I}_T^\g$ in \eqref{3:Ib} coincides with the functional
${\widetilde {\mc I}}_T^\g$  on the whole
space $ D([0,T], \mc M)$.

Fix $\pi\in  D([0,T], \mc M)$. From Lemma \ref{compar1}, we have
\begin{equation}\label{lb-1}
 {\mc I}_T^\g (\pi )\, \le\, {\widetilde {\mc I}}_T^\g(\pi )\, .
\end{equation}
Conversely, suppose that ${\mc I}_T^\g (\pi )< \infty$, then
by Lemma \ref{rep0}, there exists $F\in H^1_0(\sigma(\pi))$
such that $\pi(t,du)=\rho(t,u)du$ and $\rho$ solves the equation \eqref{eq:V} with $\mb V=\nabla F$.
Let $\mb W^F$ the weak solution of the equation
\eqref{rep2} with $\mb U=\nabla F$, it is easy to check that $(\mb W^F, \pi)\in \mf A_\g$ and
\begin{equation}\label{lb-2}
{\widetilde {\mc I}}_T^\g(\pi )\, \le\,  \mc J^\g_T (\mb W^F, \pi)\, =\, {\mc I}_T^\g (\pi )\, .   
\end{equation}

We deduce from \eqref{lb-1} and \eqref{lb-2}, that for each $\pi\in  D([0,T], \mc M)$,
${\mc I}_T^\g (\pi ) <+\infty$ if and only if ${\widetilde {\mc I}}_T^\g(\pi ) <+\infty$
and then ${\widetilde {\mc I}}_T^\g(\pi )\, =\, {\mc I}_T^\g (\pi )$ which concludes the proof of \eqref{empi-dev}.
\qed

\end{document}